\documentclass{amsart}
\usepackage{ytableau}
\usepackage{tikz-cd}
\usepackage{tikz}
\usepackage{amssymb,amsmath,amsthm}
\usepackage{mathtools}

\DeclareMathOperator{\D}{\mathrm{\Delta}}

\newtheorem{theorem}{Theorem}[section]

\newtheorem{proposition}[theorem]{Proposition}
\newtheorem{corollary}[theorem]{Corollary}
\newtheorem{definition}[theorem]{Definition}
\theoremstyle{remark}
\newtheorem{remark}[theorem]{Remark}

\numberwithin{equation}{subsection}

\begin{document}

	\title{On integral $\mathrm{Ext^2}$ between certain Weyl modules of $\mathrm{GLn}$}
	\author{Maria Metzaki}
	\address{Department of Mathematics, University of Athens}
	\curraddr{}
	\email{mmetzaki@math.uoa.gr}
	\thanks{}
	\subjclass[2020]{20G05, 20C30, 05E10}
	\keywords{Weyl modules, general linear group, extension groups}
	\date{November 1, 2024}.
	\dedicatory{}
	
\begin{abstract} Consider partitions of the form $\lambda=(a,1^b)$ and $\mu=(a+1,b-1)$,\\ where $a+1>b-1$. In this paper, we determine the extension groups $\mathrm{Ext}_A^2(K_{\lambda}F,K_{\mu}F)$, where $F$ is a free $\mathbb{Z}-$module of finite rank $n$, $K_{\lambda}F$ and $K_{\mu}F$ are the Weyl modules of the general linear group $GL_n(\mathbb{Z})$ corresponding to $\lambda$ and $\mu$, respectively, $A=S_\mathbb{Z}(n,r)$ is the integral Schur algebra and $r=a+b$.
\end{abstract}
\maketitle
	
\section{introduction}This paper concerns polynomial representations of the integral general linear group $G=GL_n(\mathbb{Z})$. For a partition $\lambda$, let $K_{\lambda}F$ denote the Weyl module of $G$ of highest weight $\lambda$. One of the important problems in the theory is the determination of the extension groups $\mathrm{Ext}_A^i(K_{\lambda}F,K_{\mu}F)$. For example, the dimensions of the modular extensions may be obtained through torsion and restriction of integral extensions using the universal coefficient theorem. Jantzen's fundamental sum formula can be viewed and proved via integral extension groups \cite{AK}. 
	
There are not many cases where explicit computations of integral extension groups between Weyl modules have been carried out. It is known that these are finite abelian groups. In \cite{AB}, the $GL_2$ case was computed for $i=1$. In \cite{BF} and \cite{C}, the $GL_3$ case was computed, when $\lambda$ and $\mu$ differ by a multiple of a single root, both for $i=1$. In \cite{Ak}, the situation $\lambda=(1^a), \mu=(a)$ was studied. In \cite{Ku2}, the case where $\lambda$ and $\mu$ are any partitions differing by a single root was settled. The situation where both $\lambda$ and $\mu$ are hooks was studied in \cite{MS} and \cite{S} for $i=1,d$ and $i=2$, respectively, where $d$ is the highest possible degree with $\mathrm{Ext}_A^d(K_{\lambda}F,K_{\mu}F) \neq 0$. In all of the above cases, the extension groups turn out to be cyclic.
	
Let $\lambda = (a,1^b)$ be a hook and let  $\mu=(a+1,b-1)$, where $a+1>b-1$. In this paper we determine explicitly $\mathrm{Ext}_A^2(K_{\lambda}F,K_{\mu}F)$. Our approach utilizes presentation matrices that we determine from the projective resolutions of $K_{\lambda}F$ found in \cite{Ma}. We compute the invariant factors of these matrices. The results show that the above extension groups are also cyclic.

\section{Notation and Recollections}In this section, we recall basic facts from the polynomial representation theory of $G=GL_n(\mathbb{Z})$ and we establish notation.
	
Let $F$ be a free $\mathbb{Z}-$module of finite rank $n$. We fix a basis of $F$, and we have an identification of general linear groups  $GL(F)=G$. 
   
Let $A=S_\mathbb{Z}(n,r)$ be the integral Schur algebra of degree $r$ corresponding to $G$ {\cite[section 2.3]{Gr}}. We will be working with homogeneous polynomial representations of $G$ of degree $r$, or equivalently, according to {\cite[section 2.4]{Gr}}, we will be working in the category of $A-$modules. 
    
By $\wedge(n,r)$ we denote the set of sequences $\alpha=(\alpha_1,\dots,\alpha_n)$
of non-negative integers with weight $\vert\alpha\vert=\alpha_1+\ldots+ \alpha_n=r$. By $\wedge^+(n,r)$ we donote the subset of $\wedge(n,r)$ consisting of sequences  $\lambda=(\lambda_1,\dots,\lambda_n)$ such that $\lambda_1\geq\lambda_2\geq\ldots\geq\lambda_n$.The elements of $\wedge^+(n,r)$ are referred to as partitions of $r$ with at most $n$ parts. A partition of the form $h=(a,1^b)$ is called a hook.
    
The divided power algebra  $DF = \oplus_{i\geq 0}D_iF$ of $F$ {\cite[section I.4]{ABW}}, is defined as the graded dual of the Hopf algebra  $S(F^*)$, where $F^*$ is the linear dual of $F$ and  $S(F^*)$ is the symmetric algebra of $F^*$. If $f\in F$ and $p,q$ are non-negative integers, we have the relation
\begin{equation*}
f^{(p)}f^{(q)}=\tbinom{p+q}{q}f^{(p+q)},
\end{equation*}
where $\tbinom{p+q}{q}$ is the indicated binomial coefficient. By  $\Delta$ (respectively, $m$)  we denote the indicated component of the diagonalization (respectively, multiplication) map of the Hopf algebra $DF$. For each   $\alpha=(\alpha_1, \ldots,\alpha_n)\in\wedge(n,r)$, the module $D_{\alpha}F=D_{\alpha_1}F\otimes \ldots \otimes D_{\alpha_n}F$  is projective over $A$ {\cite[Proposition 2.1]{AB}}. Throughout this paper, all tensor products are over the integers  $\mathbb{Z}$. Occasionally, for simplicity, we denote by $D(\alpha)$ or $D(\alpha_1,\dots,\alpha_n)$  the tensor product $D_{\alpha_1}F\otimes\ldots\otimes D_{\alpha_n}F$. 
    
Let $\lambda\in\wedge^+(n,r)$ and $d'_\lambda$ is the map in {\cite[Definition II.1.3]{ABW}}. By $K_{\lambda}F$  we denote the corresponding Weyl module for  $A$, defined as the image $\mathrm{Im}d'_\lambda$ of $d'_\lambda$.
    
Let $\{f_1,\ldots,f_n\}$ be an ordered basis for the free $\mathbb{Z}-$module $F$. For simplicity, we  denote $f_i$ by $i$. For a partition $\mu=(\mu_1,\mu_2)\in\wedge^+(n,r)$, a tableau of shape $\mu$ is a filling of the diagram of  $\mu$ with entries from the set $\{1,\ldots,n\}$. Such a tableau is called standard if the entries are weakly increasing across the rows from left to right and strictly increasing in the columns from top to bottom. The terminology used in {\cite[Definition II.3.2]{ABW}}, is co-standard. The content of a tableau $T$ is the tuple $\alpha=(\alpha_1,\ldots,\alpha_n)$, where $\alpha_i$ is the number of appearances of the entry $i$ in $T$.
    
To each tableau $T$ of shape $\mu=(\mu_1,\mu_2)\in\wedge^+(n,r)$, we associate an element 
\begin{equation*}
\mathrm{X}_T=\mathrm{X}_T(1)\otimes \mathrm{X}_T(2)\in D_{\mu_1}F\otimes D_{\mu_2}F=D_{\mu}F,
\end{equation*} 
where $\mathrm{X}_T(i)=1^{(a_{i1})}\dots n^{(a_{in})}$ and $a_{ij}$ is the number of appearances of $j$ in the $i-th$ row of $T$ {\cite[section II.2, p.224]{ABW}}. 

\begin{theorem}\label{thm21}{\cite[Theorem II.3.16]{ABW}}
Let $\mu=(\mu_1,\mu_2)\in\wedge^+(n,r)$. The set $\{d'_\mu(\mathrm{X}_T): T \text{ is a standard tableau of shape } \mu\}$ is a basis of the $\mathbb{Z}-$module $K_{\mu}F$.
\end{theorem}

Let $\alpha=(\alpha_1,\ldots,\alpha_n)\in\wedge(n,r)$ and let $M$ be an $A-$module. By $M_\alpha$ we denote the $\alpha-$weight subspace of $M$, defined in {\cite[section 3.2]{Gr}}. According to {\cite[section 2, p.178, eqn(11)]{AB}}, we may identify the $\mathbb{Z}-$module $\mathrm{Hom}_{A}(D_{\alpha}F,M)$ with the $\mathbb{Z}-$module $M_\alpha$. Especially if $M=K_{\mu}F$, where $\mu=(\mu_1,\mu_2)\in\wedge^+(n,r)$, the $\mathbb{Z}-$module $\mathrm{Hom}_{A}(D_{\alpha}F,K_{\mu}F)$ may be identified with the $\mathbb{Z}-$module $\left(K_{\mu}F\right)_{\alpha}$. More specifically, we have the following

\begin{proposition}\label{pro22} 
There exists an isomorphism of $\mathbb{Z}-$modules
\begin{equation*}
\Phi_\alpha:\mathrm{Hom}_{A}(D_{\alpha}F,K_{\mu}F)\xrightarrow{\cong}\left(K_{\mu}F\right)_{\alpha},
\end{equation*}
where $\Phi_\alpha(\tau)=\tau(1^{(\alpha_1)}\otimes 2^{(\alpha_2)}\otimes\ldots\otimes n^{(\alpha_n)})$, for each $\tau\in\mathrm{Hom}_{A}(D_{\alpha}F,K_{\mu}F)$.
\end{proposition}

Theorem \ref{thm21} and the definition of $\left(K_{\mu}F \right)_{\alpha}$ yield the following
\begin{theorem}\label{thm23} 
Let $\alpha=(\alpha_1,\ldots,\alpha_n)\in\wedge(n,r)$ and $\mu=(\mu_1,\mu_2)\in\wedge^+(n,r).$ The~ set 
\begin{equation*}
\{d'_\mu(\mathrm{X}_T): T \text{ is a standard tableau of shape } \mu \text{ and content } \alpha\}
\end{equation*}
is a basis of the $\mathbb{Z}-$module  $\left(K_{\mu}F\right)_{\alpha}$ and is called the standard basis of $\left(K_{\mu}F\right)_{\alpha}$.
\end{theorem}

Throughout this paper, we use the following notation of matrices:
\begin{itemize}
\item $O_{m\times n}$ is the  $m\times n$ zero matrix and $O_n$ is the $n\times n$ zero matrix
\item $I_n$ is the $n\times n$ identity matrix
\item $E_{i,j}$ is a matrix such that the $(i,j)-$element is equal to 1,  and each of the rest elements is equal to zero
\item $E_i$ is a square matrix such that the $(i,i)-$element is equal to 1,  and each of the rest elements is equal to zero
\item the missing entries of any matrix are equal to zero
\item $\Gamma_l$ is the $l-th$ row of a matrix or a block of  matrices
\item $\Sigma_t$ is the $t-th$ column of a matrix or a block of matrices.
\end{itemize}

Finally, let $e(m)=\{e_1,\ldots,e_m\}$ and $\epsilon(n)=\{\epsilon_1,\ldots,\epsilon_n\}$ be the standard basis of the $\mathbb{Z}-$module
$\mathbb{Z}^m$ and $\mathbb{Z}^n$, respectively. If $f:\mathbb{Z}^m\to\mathbb{Z}^n$ is a homomorphism of $\mathbb{Z}-$modules, then
$(f : e(m), \epsilon(n))\in\mathbb{Z}^{n\times m}$  is the matrix of $f$ in terms of $e(m)$ and $\epsilon(n)$.
	 
\section{$\mathrm{Ext}_A^2(K_{\lambda}F,K_{\mu}F)$ and its presentation matrix $M$}
We consider partitions of the form $\lambda = (a,1^b)$ and $\mu = (a+1,b-1)$, where  $a+1>b-1$. Let $F$ be a free $\mathbb{Z}-$module of finite rank $n$. Let $A=S_\mathbb{Z}(n,r)$ be the integral Schur algebra of degree $r=\vert\lambda\vert=\vert\mu\vert=a+b$. Suppose $n\geq r$. 
\begin{proposition}\label{pro31} 
For $b=2$, $\mathrm{Ext}_A^2(K_{\lambda}F,K_{\mu}F)=0$.
\end{proposition}
\begin{proof}
{\cite[the Remark after Theorem 4.1, for $i=2, k=1$ and $h=(a,1,1)$]{MS}}.
\end{proof}

Now, let $b\geq3$. According to {\cite[Theorem 1]{Ma}}, there is a finite projective resolution of length $b$
\begin{align*}
0\xrightarrow{}P_b(\lambda)\xrightarrow{d_b(\lambda)}P_{b-1}(\lambda)\xrightarrow{d_{b-1}(\lambda)}\cdots\xrightarrow{d_2(\lambda)}P_1(\lambda)
\xrightarrow{d_1(\lambda)}P_0(\lambda)
\xrightarrow{d'_\lambda}K_{\lambda}F\xrightarrow{}0
\end{align*}
of the Weyl module $K_{\lambda}F$ over $A$. We note that $P_0(\lambda)=D_{\lambda}F$ {\cite[section 2.3]{MS}}. Applying the contravariant functor $\mathrm{Hom}_{A}(-,K_{\mu}F)$ to the above deleted resolution, we obtain the complex
\begin{align*}
&0\xrightarrow{}\mathrm{Hom}_{A}(P_0(\lambda),K_{\mu}F)\xrightarrow{\left(d_1(\lambda)\right)_{*}}\mathrm{Hom}_{A}(P_1(\lambda),K_{\mu}F)\xrightarrow{\left(d_2(\lambda)\right)_{*}}\\
&\xrightarrow{\left(d_2(\lambda) \right)_{*}}\mathrm{Hom}_{A}(P_2(\lambda),K_{\mu}F)\xrightarrow{\left(d_3(\lambda)\right)_{*}}\dots \xrightarrow{\left(d_b(\lambda) \right)_{*}}\mathrm{Hom}_{A}(P_b(\lambda),K_{\mu}F)\xrightarrow{}0
\end{align*} 
where $\left(d_i(\lambda)\right)_{*}:=\mathrm{Hom}_{A}(d_i(\lambda),K_{\mu}F)$. From {\cite[last paragraph of section 2.3]{MS}}, we obtain
\begin{proposition}\label{pro32}
The $\mathbb{Z}-$module $\mathrm{Ext}_A^2(K_{\lambda}F,K_{\mu}F):=\dfrac{\operatorname{ker}\left(d_3(\lambda)\right)_{*}}{\operatorname{Im}\left(d_2(\lambda)\right)_{*}}$ is isomorphic to the torsion submodule $T(coker\left(d_2(\lambda)\right)_{*})$ of the cokernel of the map $\left(d_2(\lambda)\right)_{*}$.
\end{proposition}

Let $d_1=\pm\mathrm{g.c.d.}\{2,\omega\}$, where $\omega=a+2$, if $b$ is even and $\omega=a+1$, if $b$ is odd, $d_2=\pm\mathrm{g.c.d.}\{6, 2(a+1), 3(a+2), (a+1)(a+2)\}$ and $d_3=\pm\mathrm{g.c.d.}\{ 6, 2(a+2), 3(a+1),$\\$ (a+1)(a+2)\}$. In this paper, we will prove the following
\begin{theorem}\label{thm33}
Let $\lambda=(a,1^b)$ and $\mu=(a+1,b-1)$, where $b\geq3$ and $a+1>b-1$. Then
\begin{enumerate}
\item for $b=3$ or $b\geq 6$, $\mathrm{Ext}_A^2(K_{\lambda}F, K_{\mu}F)\cong\mathbb{Z}_{d_1}=
\begin{cases}\mathbb{Z}_2,&\text{if } a \equiv b \bmod 2 \\ 0, &\text{if } a \not\equiv b\bmod 2,\end{cases}$\vspace*{0,15cm}
\item for $b=4$, $\mathrm{Ext}_A^2(K_{\lambda}F, K_{\mu}F)\cong\mathbb{Z}_{d_2}=\begin{cases}
\mathbb{Z}_6,&\text{if } a \text{ is even and } 3\mid a+1 \\
\mathbb{Z}_2, &\text{if } a \text{ is even and } 3\nmid a+1 \\
\mathbb{Z}_3, &\text{if } a \text{ is odd and } 3 \mid a+1 \\
0, &\text{if } a \text{ is odd and } 3 \nmid a+1, \\
\end{cases}$\vspace*{0,15cm}
\item for $b=5$, $\mathrm{Ext}_A^2(K_{\lambda}F, K_{\mu}F)\cong\mathbb{Z}_{d_3}=\begin{cases}
\mathbb{Z}_6, &\text{if } a \text{ is odd and } 3\mid a+2 \\
\mathbb{Z}_2, &\text{if } a \text{ is odd and } 3\nmid a+2 \\
\mathbb{Z}_3, &\text{if } a \text{ is even and } 3 \mid a+2 \\
0, &\text{if } a \text{ is even and } 3 \nmid a+2.\end{cases}$
\end{enumerate}
\end{theorem}

To determine $\mathrm{Ext}_A^2(K_{\lambda}F, K_{\mu}F)$, we need to study $\mathrm{Hom}_{A}(P_{\star}(\lambda),K_{\mu}F)$,where $\star =1,2$.
\subsection{}$\boldsymbol{\mathrm{Hom}_{A}(P_1(\lambda),K_{\mu}F).}$ From {\cite[section 2.3]{MS}}, we know that
\begin{equation*} 
P_1(\lambda)=\oplus D(m_1,\ldots,m_b), 
\end{equation*}
where the sum ranges over all sequences $(m_1,\ldots,m_b)$ of positive integers such that $m_1+\ldots+m_b=a+b$  \text{ and }  $a\leq m_1\leq a+1$. It follows that  $(m_1,\dots,m_b)=(a+1,1^{b-1})$ or $(m_1,\dots,m_b)=(a,1^{i-2},2,1^{b-i})$, where $2\leq i\leq b$. We set
\begin{equation}\label{eq311}
v_1:=(a+1,1^{b-1}) \text{ and } v_i:=(a,1^{i-2},2,1^{b-i}),\text{ where } 2\leq i\leq b.
\end{equation}
The ordering of $v_i$ is defined as $v_1<v_2<\ldots<v_b$ and may be identified with the reverse lexicographic order. Now, we have
\begin{equation}\label{eq312}
P_1(\lambda)=\mathop{\oplus}_{i=1}^{b}D_{v_i}F,
\end{equation}
consequently
\begin{equation}\label{eq313}
\mathrm{Hom}_{A}(P_1(\lambda),K_{\mu}F)\cong \mathop{\oplus}_{i=1}^{b}\mathrm{Hom}_{A}(D_{v_i}F,K_{\mu}F).
\end{equation}
We note that (\ref{eq313}) induces the embedding
\begin{equation}\label{eq314}
 i(v_i):\mathrm{Hom}_{A}(D_{v_i}F,K_{\mu}F)\to\mathrm{Hom}_{A}(P_1 (\lambda),K_{\mu}F),\text{ where }  1\leq i\leq b.
\end{equation}
$\bullet$ $\boldsymbol{\mathrm{Hom}_{A}( D_{v_1}F,K_{\mu}F).}$ From Proposition \ref{pro22} for $\alpha=v_1$, we obtain the isomorphism of $\mathbb{Z}-$modules 
\begin{equation*}
\Phi_{v_1}:\mathrm{Hom}_{A}(D_{v_1}F,K_{\mu}F)\xrightarrow{\cong}\left(K_{\mu}F\right)_{v_1}.
\end{equation*}
 
There is exactly one standard tableau $S_1$ of shape $\mu = (a+1,b-1)$ and content $v_1=(a+1,1^{b-1})$, such that
\begin{equation*}
\mathrm{X}_{S_1}=1^{(a+1)}\otimes2\dots b.
\end{equation*}
Therefore, according to Theorem \ref{thm23}, the set $\{d'_\mu(\mathrm{X}_{S_1})\}$ is a basis of the $\mathbb{Z}-$module $\left(K_{\mu}F\right)_{v_1}$. We set
\begin{equation}\label{eq315}
\sigma_1:=d'_\mu\circ(1\otimes m)\in\mathrm{Hom}_{A}(D_{v_1}F,K_{\mu}F),
\end{equation}
where $m:(D_1F)^{\otimes(b-1)}\to D_{b-1}F$ is the indicated component of the multiplication map of the Hopf algebra $DF$. It is easy to verify that
\begin{equation}\label{eq316}
\Phi_{v_1}(\sigma_1)=
d'_\mu(\mathrm{X}_{S_1}),
\end{equation}
consequently
\begin{equation}\label{eq317}
\text{the set } \{\sigma_1\} \text{ is a basis of the } \mathbb{Z}-\text{module } \mathrm{Hom}_{A}(D_{v_1}F,K_{\mu}F).
\end{equation}
It follows that there exists an isomorphism of $\mathbb{Z}-$modules 
\begin{equation}\label{eq318}
\phi_1:\mathrm{Hom}_{A}(D_{v_1}F,K_{\mu}F)\xrightarrow{\cong}\mathbb{Z},\text{ such that }\phi_1(\sigma_1)=1.
\end{equation}
$\bullet$ $\boldsymbol{\mathrm{Hom}_{A}(D_{v_i}F, K_{\mu}F),\ 2\leq i\leq b.}$ Let $2\leq i\leq b$. We consider the $\mathbb{Z}-$module isomorphism $\Phi_{v_i}:\mathrm{Hom}_{A}(D_{v_i}F,K_{\mu}F)\xrightarrow{\cong}\left(K_{\mu}F\right)_{v_i}$.

There are exactly $b-1$ standard tableaux $S_{ij}$, where $2\leq j\leq b$, of shape $\mu =(a+1,b-1)$ and content $v_i=(a,1^{i-2},2,1^{b-i})$, such that
\begin{equation}\label{eq319}
\mathrm{X}_{S_{ij}}=\begin{cases}
1^{(a)}j\otimes 2\dots\widehat{j}\dots (i-1)i^{(2)}(i+1)\dots b,&\text{ if } 2\leq j\leq i-1\\
1^{(a)}i\otimes 2\dots b,&\text{ if }j=i\\
1^{(a)}j\otimes 2\dots (i-1)i^{(2)}(i+1)\dots \widehat{j}\dots b,&\text{ if } i+1\leq j\leq b,
\end{cases}
\end{equation}
where  $\widehat{j}$ means that $j$ is omitted. Therefore, the set $\{d'_\mu(\mathrm{X}_{S_{ij}}):2\leq j\leq b\}$ is a basis of the $\mathbb{Z}-$module $\left(K_{\mu}F\right)_{v_i}$.

For $2\leq j\leq b$ with  $j\neq i$, define the map $\tau_j\in\mathrm{Hom}_{A}(D_{v_i}F,D_{\mu}F)$
as follows: for
\begin{equation*}
x_1\otimes x_2\otimes\ldots\otimes x_b \in D_{v_i}F, \;\tau_j(x_1\otimes x_2\otimes\ldots\otimes x_b)=x_1x_j\otimes x_2\ldots\widehat{x_j}\ldots x_b, \end{equation*}
where $x_1x_j$ and $x_2\ldots\widehat{x_j}\ldots x_b$  denote the product in the divided power algebra $DF$. Likewise, we define the map $\tau_i\in\mathrm{Hom}_{A}(D_{\lambda}F,D_{\mu}F)$. 
Now, we set
\begin{equation}\label{eq3110}
\sigma_{ij}:=\begin{cases}
d'_\mu\circ\tau_j\in\mathrm{Hom}_{A}(D_{v_i}F,K_{\mu}F),\text{ if } 2\leq j\leq b \text{ and } j\neq i\\
d'_\mu\circ\tau_i\circ(1^{\otimes(i-1)}\otimes \Delta\otimes1^{\otimes(b-i)})\in\mathrm{Hom}_{A}(D_{v_i}F,K_{\mu}F),\text{ if } j=i, 
\end{cases}
\end{equation} 
where $\Delta:D_2F\to D_1F\otimes D_1F$ is the indicated component of the diagonalization map of the Hopf algebra $DF$. It is easy to verify that
\begin{equation}\label{eq3111}
\Phi_{v_i}(\sigma_{ij})=d'_\mu(\mathrm{X}_{S_{ij}}),\text{ where } 2\leq j\leq b,
\end{equation}
consequently 
\begin{equation}\label{eq3112}
\text{the set } \{\sigma_{ij}: 2\leq j\leq b \} \text{ is a basis of the }
\mathbb{Z}-\text{module } \mathrm{Hom}_{A}(D_{v_i}F,K_{\mu}F). 
\end{equation}
It follows that there exists an isomorphism of $\mathbb{Z}-$modules
\begin{equation}\label{eq3113}
\phi_i:\mathrm{Hom}_{A}(D_{v_i}F,K_{\mu}F)\xrightarrow{\cong}\mathbb{Z}^{b-1}, \text{ such that } \phi_i(\sigma_{ij})=e_{j-1},  \text{ for }   2\leq j\leq b.
\end{equation}

From (\ref{eq313}), (\ref{eq314}), (\ref{eq317}), (\ref{eq318}), (\ref{eq3112}) and (\ref{eq3113}), we conclude that there exists an isomorphism of $\mathbb{Z}-$modules
\begin{equation}\label{eq3114}
\phi:=\mathop{\oplus}\limits_{i=1}^{b}\phi_i:\mathrm{Hom}_{A}(P_1(\lambda),K_{\mu}F)\xrightarrow{\cong}\mathbb{Z}^t,\text{ where } t=1+(b-1)^2,
\end{equation}
and $\{\overline{\sigma_1},\overline{\sigma_{ij}}: 2\leq i, j\leq b\}$ is a basis of the $\mathbb{Z}-$module $\mathrm{Hom}_{A}(P_1(\lambda),K_{\mu}F)$, where $\overline{\sigma_1}:=i(v_1)(\sigma_1)= {(\sigma_1, 0^{b-1})}$ and   $\overline{\sigma_{ij}}:= i(v_i)(\sigma_{ij})= {(0^{i-1}, \sigma_{ij}, 0^{b-i})}$.

\subsection{$\boldsymbol{\mathrm{Hom}_{A}(P_2(\lambda),K_{\mu}F)}$} We proceed similarly to the case $\mathrm{Hom}_{A}(P_1(\lambda),K_{\mu}F)$. Therefore, we set
\begin{align}\label{eq321}
&u_1:=(a+2,1^{b-2}),\; u_i:=(a+1,1^{i-2},2, 1^{b-1-i}),\text{ where } 2\leq i\leq b-1,\\\nonumber &w_i:=(a,1^{i-2},3,1^{b-1-i}), \text{ where } 2\leq i\leq b-1, \text{ and }\\\nonumber& w_{ij}:=(a,1^{i-2},2,1^{j-1-i},2,1^{b-1-j}), \text{ where } 2\leq i<j\leq b-1. 
\end{align} 
The ordering of $u_i,w_i$ and $w_{ij}$ is defined as $u_1<u_2<\ldots<u_{b-1}<w_2<w_{23}<\\<\ldots <w_{2b-1}<\ldots <w_i<w_{ii+1}<\ldots< w_{ib-1}<\ldots < w_{b-2}<w_{b-2b-1}<w_{b-1}$ and may be identified with the reverse lexicographic order. Now, we have 

\begin{equation}\label{eq322}
P_2(\lambda)=\left(\mathop{\oplus}_{i=1}^{b-1}D_{u_i}F\right)\oplus\left(\mathop{\oplus}_{i=2}^{b-1}\left(D_{w_i}F\oplus\left( \mathop{\oplus}_{j=i+1}^{b-1}D_{w_{ij}}F\right) \right) \right), \text{ and}
\end{equation}
\begin{align}\label{eq323}
\mathrm{Hom}_{A}\left(P_2(\lambda),K_{\mu}F\right)\cong\left(\mathop{\oplus}_{i=1}^{b-1}\mathrm{Hom}_{A}\left(D_{u_i}F,K_{\mu}F\right)\right)\oplus\left(\mathop{\oplus}_{i=2}^{b-1}\left(\mathrm{Hom}_{A}\left(D_{w_i}F,K_{\mu}F\right)\oplus\right.\right.&\\\nonumber
\oplus\left.\left.\left(\mathop{\oplus}_{j=i+1}^{b-1}\mathrm{Hom}_{A}\left(D_{w_{ij}}F,K_{\mu}F \right)\right)\right)\right).
\end{align}
Obviously, (\ref{eq323}) induces the projections
\begin{align}\label{eq324}
&pr(u_i):\mathrm{Hom}_{A}(P_2(\lambda),K_{\mu}F)\to\mathrm{Hom}_{A}(D_{u_i}F,K_{\mu}F),\text{ where } 1\leq i \leq b-1,\\\nonumber
&pr(w_i):\mathrm{Hom}_{A}(P_2(\lambda),K_{\mu}F)\to\mathrm{Hom}_{A}(D_{w_i}F,K_{\mu}F),\text{ where } 2\leq i \leq b-1, \text{ and
}\\\nonumber&pr(w_{ij}):\mathrm{Hom}_{A}(P_2(\lambda),K_{\mu}F)\to\mathrm{Hom}_{A}(D_{w_{ij}}F,K_{\mu}F), \text{ where } 2\leq i<j \leq b-1.
\end{align}
$\bullet$ $\boldsymbol{\mathrm{Hom}_{A}(D_{u_1}F,K_{\mu}F).}$ There is no standard tableau of shape  $\mu=(a+1,b-1)$ and content $u_1=(a+2,1^{b-2})$, therefore $\left(K_{\mu}F\right)_{u_1}=O$. From the $\mathbb{Z}-$module isomorphism $\Phi_{u_1}:\mathrm{Hom}_{A}(D_{u_1}F,K_{\mu}F)\xrightarrow{\cong}\left( K_{\mu}F\right)_{u_1},$ we conclude that
\begin{equation}\label{eq325}
\mathrm{Hom}_{A}(D_{u_1}F,K_{\mu}F)=O.
\end{equation}
$\bullet$ $\boldsymbol{\mathrm{Hom}_{A}(D_{u_i}F,K_{\mu}F),\ 2\leq i\leq b-1.}$ Let $2\leq i\leq b-1$. We proceed similarly to the case  $\mathrm{Hom}_{A}(D_{v_1}F,K_{\mu}F)$. Therefore, we conclude that
\begin{equation}\label{eq326}
\mathrm{X}_{P_i}=1^{(a+1)}\otimes2\ldots (i-1)i^{(2)}(i+1)\ldots(b-1),
\end{equation}
where $P_i$ is the unique standard tableau of shape $\mu=(a+1,b-1)$ and content $u_i=(a+1,1^{i-2},2,1^{b-1-i})$, and we set
\begin{equation}\label{eq327}
\pi_i:=d'_\mu\circ(1\otimes m)\in\mathrm{Hom}_{A}(D_{u_i}F,K_{\mu}F),
\end{equation}
where $m:(D_1F)^{\otimes(i-2)}\otimes D_2F\otimes(D_1F)^{\otimes(b-1-i)}\to D_{b-1}F$ is the indicated component of the multiplication map of the Hopf algebra $DF$. We  easily  verify that
\begin{equation}\label{eq328}
\Phi_{u_i}(\pi_i)=d'_\mu(x_{P_i}),\text{ and }
\end{equation}
there exists an isomorphism of $\mathbb{Z}-$modules
\begin{equation}\label{eq329}
\psi_i:\mathrm{Hom}_{A}(D_{u_i}F,K_{\mu}F)\xrightarrow{\cong}\mathbb{Z}, \text{ such that }\psi_i(\pi_i)=1.
\end{equation}
$\bullet$ $\boldsymbol{\mathrm{Hom}_{A}(D_{w_i}F,K_{\mu}F),\ 2\leq i\leq b-1.}$ Let $2\leq i\leq b-1$. We proceed similarly to the case $\mathrm{Hom}_{A}(D_{v_i}F,K_{\mu}F),\ 2\leq i\leq b$. Therefore, we conclude that
\begin{equation}\label{eq3210}
\mathrm{X}_{P_{ij}}=\begin{cases}
1^{(a)}j\otimes2\ldots\widehat{j}\ldots (i-1)i^{(3)}(i+1)\ldots(b-1),&\text{ if } 2\leq j\leq i-1\\
1^{(a)}i\otimes2\ldots(i-1)i^{(2)}(i+1)\ldots (b-1),&\text{ if }j=i\\
1^{(a)}j\otimes2\ldots(i-1)i^{(3)}(i+1)\ldots \widehat{j}\ldots(b-1),&\text{ if } i+1\leq j\leq b-1,
\end{cases}
\end{equation}
where $P_{ij}$ is a standard tableau of shape $\mu=(a+1,b-1)$ and content $w_i=(a,1^{i-2},3,1^{b-1-i})$, and we set 
\begin{equation}\label{eq3211}
\pi_{ij}:=\begin{cases}
d'_\mu\circ t_j\in\mathrm{Hom}_{A}(D_{w_i}F,K_{\mu}F),\text{ if } 2\leq j\leq b-1\text{ and } j\neq i\\
d'_\mu\circ t_i\circ (1^{\otimes(i-1)}\otimes \Delta\otimes 1^{\otimes(b-1-i)})\in\mathrm{Hom}_{A}(D_{w_i}F,K_{\mu}F),\text{ if } j=i,
\end{cases}
\end{equation}
where $t_j\in\mathrm{Hom}_{A}(D_{w_i}F,D_{\mu}F)$, for $2\leq j\leq b-1$ with $j\neq i$, and $t_i\in\mathrm{Hom}_{A}(D_{v_{i+1}}F,\\D_{\mu}F)$ are defined similarly to the maps $\tau_j$ [see before (\ref{eq3110})], while $\Delta:D_3F\to D_1F\otimes D_2F$ is the indicated component of the diagonalization map of the Hopf algebra $DF$. We easily verify that 
\begin{equation}\label{eq3212}
\Phi_{w_i}(\pi_{ij})=d'_\mu(\mathrm{X}_{P_{ij}}),\text{ where } 2\leq j\leq b-1, 
\end{equation}
and there exists an isomorphism of $\mathbb{Z}-$modules
\begin{equation}\label{eq3213}
\theta_i:\mathrm{Hom}_{A}(D_{w_i}F,K_{\mu}F)\xrightarrow{\cong}\mathbb{Z}^{b-2}, \text{ such that } \theta_i(\pi_{ij})=\epsilon_{j-1}, \text{ for }  2\leq j\leq b-1.
\end{equation}
$\bullet$ $\boldsymbol{\mathrm{Hom}_{A}(D_{w_{ij}}F,K_{\mu}F),\ 2\leq i<j\leq b-1.}$ Let $2\leq i<j\leq b-1$. We proceed similarly to the case  $\mathrm{Hom}_{A}(D_{v_i}F,K_{\mu}F),\ 2\leq i\leq b$. Therefore, we conclude that
\begin{equation}\label{eq3214}
\mathrm{X}_{P_{ijh}}=
\end{equation}
{\small $\begin{cases}
1^{(a)}h\otimes2\ldots\widehat{h}\ldots (i-1)i^{(2)}(i+1)\ldots(j-1)j^{(2)}(j+1)\ldots(b-1),&\text{if }2\leq h\leq i-1\\
1^{(a)}i\otimes2\ldots(j-1)j^{(2)}(j+1)\ldots(b-1),&\text{if }h=i\\
1^{(a)}h\otimes2\ldots(i-1)i^{(2)}(i+1)\ldots \widehat{h}\ldots(j-1)j^{(2)}(j+1)\ldots(b-1),&\text{if }i+1\leq h\leq j-1\\
1^{(a)}j\otimes2\ldots(i-1)i^{(2)}(i+1)\ldots(b-1),&\text{if }h=j\\
1^{(a)}h\otimes2\ldots(i-1)i^{(2)}(i+1)\ldots (j-1)j^{(2)}(j+1)\ldots\widehat{h}\ldots(b-1),&\text{if }j+1\leq h\leq b-1,
\end{cases}$}\vspace*{0,15cm}
where $P_{ijh}$ is a standard tableau of shape $\mu=(a+1,b-1)$ and content $w_{ij}=(a,1^{i-2},2,1^{j-1-i},2,1^{b-1-j})$, and we set
\begin{equation}\label{eq3215}
\pi_{ijh}:=
\end{equation}
$\begin{cases}d'_\mu\circ r_h\in\mathrm{Hom}_{A}(D_{w_{ij}}F,K_{\mu}F),\text{ if }2\leq h\leq b-1 \text{ with } h\neq i \text{ and }h\neq j\\
d'_\mu\circ r_h\circ(1^{\otimes(h-1)}\otimes\Delta \otimes1^{\otimes(b-1-h)}) \in\mathrm{Hom}_{A}(D_{w_{ij}}F,K_{\mu}F),\text{if } h=i \text{ or } h=j,
\end{cases}$\vspace*{0,15cm}\\
where $r_h\in\mathrm{Hom}_{A}(D_{w_{ij}}F,D_{\mu}F)$, for $2\leq h\leq b-1$ with $h\neq i$ and $h\neq j$,
$r_i\in\mathrm{Hom}_{A}(D_{v_{j+1}}F,D_{\mu}F)$ and $r_j\in\mathrm{Hom}_{A}(D_{v_i}F,D_{\mu}F)$, are defined similarly to the maps $\tau_j$ [see before (\ref{eq3110})], while $\Delta:D_2F\to D_1F\otimes D_1F$ is the indicated component of the diagonalization map of the Hopf algebra $DF$. We easily verify that 
\begin{equation}\label{eq3216}
\Phi_{w_{ij}}(\pi_{ijh})=d'_\mu(\mathrm{X}_{P_{ijh}}),\text{ where  } 2\leq h\leq b-1,
\end{equation}
and there exists an isomorphism of $\mathbb{Z}-$modules
\begin{equation}\label{eq3217}
\theta_{ij}:\mathrm{Hom}_{A}(D_{w_{ij}}F,K_{\mu}F)\xrightarrow{\cong}\mathbb{Z}^{b-2}, \text{ such that } \theta_{ij}(\pi_{ijh})=\epsilon_{h-1}, \text{ for }  2\leq h\leq b-1.
\end{equation}

From (\ref{eq323}), (\ref{eq325}), (\ref{eq329}), (\ref{eq3213}) and (\ref{eq3217}), we conclude that there exists an isomorphism of $\mathbb{Z}-$modules \begin{equation}\label{eq3218}
\theta:\mathrm{Hom}_{A}(P_2(\lambda),K_{\mu}F)\xrightarrow{\cong}\mathbb{Z}^s,
\end{equation} 
where\; $\theta:= \displaystyle\left(\mathop{\oplus}_{i=2}^{b-1}\psi_i\right)\oplus\left(\mathop{\oplus}_{i=2}^{b-1}\left(\theta_i\oplus\left( \mathop{\oplus}_{j=i+1}^{b-1}\theta_{{ij}}\right)\right)\right)$ and  $s=(b-2)\left[ 1+\binom{b-1}{2}\right]$.
\subsection{The presentation matrix $\boldsymbol{M}$ of $\boldsymbol{\mathrm{Ext}_A^2(K_{\lambda}F, K_{\mu}F)}$} Here, we compute the matrix from which $\mathrm{Ext}_A^2(K_{\lambda}F, K_{\mu}F)$ will be determined. We consider the homomorphism of $\mathbb{Z}-$modules
\begin{equation}\label{eq331}
f=f(a,b):=\theta\circ\left(d_2 (\lambda)\right)_*\circ\phi^{-1}:\mathbb{Z}^t\to\mathbb{Z}^s,
\end{equation}
where the maps $\phi$ and $\theta$ are given by (\ref{eq3114}) and (\ref{eq3218}), respectively, and $\left(d_2 (\lambda)\right)_*$ is described before Proposition \ref{pro32}. We set
\begin{equation}\label{eq332}
M=M(a,b):=(f:e(t),\epsilon(s))\in\mathbb{Z}^{s\times t}.
\end{equation}
We recall that $s=(b-2)\left[1+\binom{b-1}{2}\right]$ and $t=1+(b-1)^2$.

From Proposition \ref{pro32}, (\ref{eq331}) and the fact that $\phi$ and  $\theta$ are isomorphisms of $\mathbb{Z}-$modules, we conclude that
\begin{equation}\label{eq333}
\mathrm{Ext}_A^2(K_{\lambda}F,K_{\mu}F)\cong T(cokerf),
\end{equation}
where $T(cokerf)$ is the torsion submodule of the cokernel of $f$. Furthermore, it is well known that
\begin{equation}\label{eq334}
T(cokerf)\cong\mathbb{Z}_{d_1}\oplus\ldots\oplus\mathbb{Z}_{d_\nu},
\end{equation}
where $d_1|d_2|\ldots|d_\nu$ are the non-zero invariant factors of the matrix  $M$.
Hence, to determine $\mathrm{Ext}_A^2(K_{\lambda}F,K_{\mu}F)$, we will initially compute matrix $M$ and then calculate its invariant factors.

For the computation of $M$, some additional definitions will be needed. For the formulation of these definitions, we will use the maps $\phi_j$ [see (\ref{eq318}) and (\ref{eq3113})], $\psi_i$ [see (\ref{eq329})], $\theta_i$ [see (\ref{eq3213})], $\theta_{ij}$ [see (\ref{eq3217})], $i(v_j)$ [see (\ref{eq314})], $pr(u_i), pr(w_i)$ and $pr(w_{ij})$ [see (\ref{eq324})]. We set $\mu(1)=1$ and $\mu(j)=b-1$, where  $2\leq j\leq b$.

\begin{definition}\label{def34}
$i)$ The homomorphism of $\mathbb{Z}-$modules
\begin{equation*}
f(u_i,v_j):\mathbb{Z}^{\mu(j)}\to\mathbb{Z}, \text{ where } 2\leq i\leq b-1 \text{ and } 1\leq j\leq b,
\end{equation*}
is defined to be the composition
\begin{align*}
&\mathbb{Z}^{\mu(j)}\xrightarrow{\phi_j^{-1}} \mathrm{Hom}_{A}(D_{v_j}F,K_{\mu}F)\xrightarrow{i(v_j)}\mathrm{Hom}_{A}(P_1(\lambda),K_{\mu}F)\xrightarrow{(d_2(\lambda))_*}\\
&\to\mathrm{Hom}_{A}(P_2(\lambda),K_{\mu}F)\xrightarrow{pr(u_i)}\mathrm{Hom}_{A}(D_{u_i}F, K_{\mu}F)\xrightarrow{\psi_i}\mathbb{Z}.
\end{align*}
$ii)$ We set 
\begin{equation*}
A(u_i,v_j):=(f(u_i,v_j):e(\mu(j)),\epsilon(1))\in\mathbb{Z}^{1\times\mu(j)}, 
\end{equation*}
where $2\leq i\leq b-1$ and  $1\leq j\leq b$.
\end{definition}

\begin{definition}\label{def35}
$i)$ The homomorphism of $\mathbb{Z}-$modules
\begin{equation*}
f(w_i,v_j):\mathbb{Z}^{\mu(j)}\to\mathbb{Z}^{b-2}, \text{ where } 2\leq i\leq b-1 \text{ and } 1\leq j\leq b,
\end{equation*}
is defined to be the composition
\begin{align*}
&\mathbb{Z}^{\mu(j)}\xrightarrow{\phi_j^{-1}} \mathrm{Hom}_{A}(D_{v_j}F,K_{\mu}F)\xrightarrow{i(v_j)}\mathrm{Hom}_{A}(P_1(\lambda),K_{\mu}F)\xrightarrow{(d_2(\lambda))_*}\\
&\to\mathrm{Hom}_{A}(P_2(\lambda),K_{\mu}F)\xrightarrow{pr(w_i)}\mathrm{Hom}_{A}(D_{w_i}F,K_{\mu}F)\xrightarrow{\theta_i}\mathbb{Z}^{b-2}.
\end{align*}
$ii)$ We set 
\begin{equation*}
B(w_i,v_j):=(f(w_i,v_j):e(\mu(j)),\epsilon(b-2))\in\mathbb{Z}^{(b-2)\times\mu(j)}, 
\end{equation*}
where  $2\leq i\leq b-1$ and $1\leq j\leq b$.
\end{definition}

\begin{definition}\label{def36}
$i)$ The homomorphism of $\mathbb{Z}-$modules
\begin{equation*}
f(w_{ij},v_h):\mathbb{Z}^{\mu(h)}\to\mathbb{Z}^{b-2}, \text{ where } 2\leq i<j\leq b-1 \text{ and } 1\leq h\leq b,
\end{equation*}
is defined to be the composition
\begin{align*}
&\mathbb{Z}^{\mu(h)}\xrightarrow{\phi_h^{-1}} \mathrm{Hom}_{A}(D_{v_h}F,K_{\mu}F)\xrightarrow{i(v_h)}\mathrm{Hom}_{A}(P_1(\lambda),K_{\mu}F)\xrightarrow{(d_2(\lambda))_*}\\
&\to\mathrm{Hom}_{A}(P_2(\lambda),K_{\mu}F)\xrightarrow{pr(w_{ij})}\mathrm{Hom}_{A}(D_{w_{ij}}F, K_{\mu}F)\xrightarrow{\theta_{ij}}\mathbb{Z}^{b-2}.
\end{align*}
$ii)$ We set 
\begin{equation*}
C(w_{ij},v_h):=(f(w_{ij},v_h):e(\mu(h)),\epsilon(b-2))\in\mathbb{Z}^{(b-2)\times\mu(h)},
\end{equation*}
where $2\leq i<j\leq b-1$ and $1\leq h\leq b$.
\end{definition}

From (\ref{eq331}), (\ref{eq332}), (\ref{eq3114}), (\ref{eq3218}) and Definitions \ref{def34} - \ref{def36}, it follows that  matrix $M$ is of the form
\begin{equation}\label{eq335}
M=M(a,b)=\begin{bmatrix}
A(u_2,v_1) & A(u_2,v_2) & \cdots & A(u_2,v_b) \\
\vdots & \vdots & & \vdots\\
A(u_{b-1},v_1) & A(u_{b-1},v_2) & \cdots & A(u_{b-1},v_b) \\
E(2,v_1) & E(2,v_2) & \cdots & E(2,v_b)\\
\vdots & \vdots & & \vdots\\
E(b-1,v_1) & E(b-1,v_2) & \cdots & E(b-1,v_b)
\end{bmatrix}\in\mathbb{Z}^{s\times t},
\end{equation}
where $s=(b-2)\left[1+\binom{b-1}{2}\right],\; t=1+(b-1)^2$,
and for $1\leq j\leq b$, it holds  
\begin{equation*}
E(i,v_j)=\begin{bmatrix}
B(w_i,v_j) \\ C(w_{ii+1},v_j)\\\vdots\\ C(w_{ib-1},v_j) \end{bmatrix}\in \mathbb{Z}^{(b-i)(b-2)\times(b-1)},\text{ if } 2\leq i\leq b-2, \text{ and }
\end{equation*}
\begin{equation*}
E(b-1,v_j)=B(w_{b-1},v_j)\in \mathbb{Z}^{(b-2)\times(b-1)}.
\end{equation*}

We claim that for $2\leq i\leq b-1$, it holds
\begin{equation}\label{eq336} 
A(u_i,v_1)=(-1)^{i+1} [2]\in\mathbb{Z}^{1\times 1}, 
\end{equation}
\begin{equation}\label{eq337}
A(u_i,v_j)=O\in\mathbb{Z}^{1\times(b-1)},  \text{ where }2\leq j\leq b \text{ and } j\neq i+1,
\end{equation}
\begin{equation}\label{eq338}
A(u_i,v_{i+1})=\begin{bmatrix}
a+1&-1&\dots&-1&-2&-1&\dots&-1
\end{bmatrix}\in\mathbb{Z}^{1\times(b-1)},
\end{equation}
where $-2$ is located at the $i-th$ column,
\begin{equation}\label{eq339}	B(w_i,v_1)=O\in\mathbb{Z}^{(b-2)\times1}, 
\end{equation}
\begin{equation}\label{eq3310}
B(w_i,v_j)=O\in\mathbb{Z}^{(b-2)\times (b-1)},\text{ where } 2\leq j\leq b  \text{ with } j\neq i \text{ and } j\neq i+1 ,
\end{equation}
\begin{equation}\label{eq3311}
B(w_i,v_i)=(-1)^{i+1}\begin{bmatrix}
3I_{i-2} & & &\\
& 2& 1& \\
& & &3I_{(b-2)-(i-1)}
\end{bmatrix}\in\mathbb{Z}^{(b-2)\times(b-1)},
\end{equation}
\begin{equation}\label{eq3312}
B(w_i,v_{i+1})=(-1)^{i+1}\begin{bmatrix}
3I_{i-2}& & & \\
& 1& 2& \\
& & &3I_{(b-2)-(i-1)}
\end{bmatrix}\in\mathbb{Z}^{(b-2)\times(b-1)},
\end{equation}
and for  $2\leq i<j\leq b-1$, it holds
\begin{equation}\label{eq3313}
C(w_{ij},v_1)=O\in\mathbb{Z}^{(b-2)\times1}, 
\end{equation}
\begin{equation}\label{eq3314}
C(w_{ij},v_h)=O\in\mathbb{Z}^{(b-2)\times (b-1)},\text{ where }2\leq h\leq b \text{ with } h\neq i \text{ and } h\neq j+1,
\end{equation}
\begin{equation}\label{eq3315}
C(w_{ij},v_i)=(-1)^{j+1}\begin{bmatrix}
2I_{j-2}& & & \\
& 1& 1& \\
& & &2I_{(b-2)-(j-1)}
\end{bmatrix}\in\mathbb{Z}^{(b-2)\times(b-1)}, \text{ and}
\end{equation}
\begin{equation}\label{eq3316}
C(w_{ij},v_{j+1})=(-1)^{i+1}\begin{bmatrix}
2I_{i-2}& & & \\
& 1& 1& \\
& & &2I_{(b-2)-(i-1)}
\end{bmatrix}\in\mathbb{Z}^{(b-2)\times(b-1)}.
\end{equation}

To prove (\ref{eq336}) - (\ref{eq3316}), we need to study the map $d_2(\lambda):P_2(\lambda)\to P_1(\lambda)$. From the definition of $d_2(\lambda)$ ({\cite[section 2.3]{MS}}), and (\ref{eq311}), (\ref{eq312}), (\ref{eq321}) and (\ref{eq322}), it follows that 
\begin{equation}\label{eq3317}
d_2(\lambda)=\left(\mathop{\oplus}_{i=1}^{b-1}d_2(u_i)\right)\oplus\left(\mathop{\oplus}_{i=2}^{b-1}\left(d_2(w_i)\oplus\left( \mathop{\oplus}_{j=i+1}^{b-1}d_2(w_{ij})\right)\right)\right),
\end{equation}
where
\begin{align}\label{eq3318}
&d_2(u_1)=d_2(\lambda,u_1):=\Delta\otimes 1^{\otimes(b-2)}:D_{u_1}F=D(a+2,1^{b-2})\to\\
&\nonumber\to D(a+1,1^{b-1})\oplus D(a,2,1^{b-2})=D_{v_1}F\oplus D_{v_2}F,
\end{align}
where $\Delta:D(a+2)\to D(a+1,1)\oplus D(a,2)$ is the indicated component of the diagonalization map of the Hopf algebra $DF$,
\begin{align}\label{eq3319}
&\text{for } 2\leq i\leq b-1,\;
d_2(u_i)=d_2(\lambda,u_i):=\\
&\nonumber\Delta_1\otimes1^{\otimes(b-2)}\oplus(-1)^{i+1}1^{\otimes(i-1)}\otimes\Delta_i\otimes 1^{\otimes(b-1-i)}:D_{u_i}F=\\
&\nonumber D(a+1,1^{i-2},2,1^{b-1-i})\to D(a,1^{i-1},2,1^{b-1-i})\oplus D(a+1,1^{b-1})=\\
&\nonumber D_{v_{i+1}}F\oplus D_{v_1}F,
\end{align}
where $\Delta_1:D(a+1)\to D(a,1)$ and  $\Delta_i:D(2)\to D(1,1)$ are the indicated components of the diagonalization map of the Hopf algebra $DF$,
\begin{align}\label{eq3320}
&\text{for }2\leq i\leq b-1,\;
d_2(w_i)=d_2(\lambda,w_i):=\\
&\nonumber(-1)^{i+1}1^{\otimes(i-1)}\otimes\Delta_i\otimes1^{\otimes(b-1-i)}:D_{w_i}F=D(a,1^{i-2},3,1^{b-1-i})\to\\
&\nonumber\to D(a,1^{i-2},2,1^{b-i})\oplus  D(a,1^{i-1},2,1^{b-1-i})=D_{v_{i}}F\oplus D_{v_{i+1}}F,
\end{align}
where $\Delta_i:D(3)\to D(2,1)\oplus D(1,2)$ is the indicated component of the diagonalization map of the Hopf algebra $DF$, and\newpage
\begin{align}\label{eq3321}
&\text{for }2\leq i<j\leq b-1,\
d_2(w_{ij})=d_2(\lambda,w_{ij}):=\\
\nonumber&(-1)^{i+1}1^{\otimes(i-1)}\otimes \Delta_i\otimes1^{\otimes(b-1-i)}\oplus(-1)^{j+1}1^{\otimes(j-1)}\otimes\Delta_j\otimes1^{\otimes(b-1-j)}:\\
\nonumber&D_{w_{ij}}F=D(a,1^{i-2},2,1^{j-1-i},2,1^{b-1-j})\to\\
\nonumber&\to D(a,1^{j-1},2,1^{b-1-j})\oplus  D(a,1^{i-2},2,1^{b-i})=D_{v_{j+1}}F\oplus D_{v_i}F,
\end{align}
where $\Delta_i=\Delta_j: D(2)\to D(1,1)$ is the indicated component of the diagonalization map of the Hopf algebra $DF$.

We recall that $\overline{\sigma_1}$ and $\overline{\sigma_{kl}}$ are given by (\ref{eq3114}), while $pr(u_i)$, $pr(w_i)$ and $pr(w_{ij})$ are given by (\ref{eq324}). Furthermore,
\begin{equation}\label{eq3322}
i(v_\alpha,v_\beta):D_{v_\alpha}F\oplus D_{v_\beta}F\to P_1(\lambda) \text{ and } pr(v_k):P_1(\lambda)\to D_{v_k}F 
\end{equation}
are the embedding and the projection, respectively, induced by (\ref{eq312}) .  Clearly,
\begin{equation}\label{eq3323}
pr(v_k)\circ i(v_\alpha,v_\beta)=\begin{cases} pr(\overline{v_\alpha},v_\beta),&\text{ if } k=\alpha\\ 
pr(v_\alpha,\overline{v_\beta}), &\text{ if } k=\beta\\
0,&\text{ if }  k\neq \alpha \text{ and } k\neq \beta,
\end{cases}
\end{equation}
where
\begin{equation}\label{eq3324}
pr(\overline{v_\alpha},v_\beta): D_{v_\alpha}F\oplus D_{v_\beta}F\to D_{v_\alpha}F \text{ and } pr(v_\alpha,\overline{v_\beta}): D_{v_\alpha}F\oplus D_{v_\beta}F\to  D_{v_\beta}F
\end{equation}
are the obvious projections. Finally, we need the subsequent

\begin{proposition}\label{pro37}
The following relations hold:
\begin{enumerate}
\item$[pr(u_i)\circ(d_2(\lambda))_*](\overline{\sigma_1})=\sigma_1\circ pr(v_{i+1},\overline{v_1})\circ d_2(u_i)$, where $1\leq i\leq b-1$
\item$[pr(w_i)\circ(d_2(\lambda))_*](\overline{\sigma_1})=0$, where $2\leq i\leq b-1$
\item$[pr(w_{ij})\circ(d_2(\lambda))_*](\overline{\sigma_1})=0$, where $2\leq i<j\leq b-1$
\item$[pr(u_i)\circ(d_2(\lambda))_*](\overline{\sigma_{kl}})=\sigma_{kl}\circ pr(v_k)\circ i(v_{i+1},v_1)\circ d_2(u_i)$, where $1\leq i\leq b-1$ and $2\leq k,l \leq b$
\item$[pr(w_i)\circ(d_2(\lambda))_*](\overline{\sigma_{kl}})=\sigma_{kl}\circ pr(v_k)\circ i(v_{i},v_{i+1})\circ d_2(w_i)$, where $2\leq i\leq b-1$ and $2\leq k,l \leq b$
\item$[pr(w_{ij})\circ(d_2(\lambda))_*](\overline{\sigma_{kl}})=\sigma_{kl}\circ pr(v_k)\circ i(v_{j+1},v_i)\circ d_2 (w_{ij})$, where $2\leq i<j\leq b-1$ and $2\leq k,l \leq b$.
\end{enumerate}
\end{proposition}

\begin{proof}[\textbf{Proof of $\boldsymbol{\mathrm{(\ref{eq336})}}$}]From Definition \ref{def34} $i)$ for $j$=1, (\ref{eq318}), (\ref{eq3114}) and Proposition \ref{pro37}(1), it follows that 
\begin{equation}\label{eq1}
f(u_i,v_1)(1)=\psi_i(\sigma_1\circ pr(v_{i+1},\overline{v_1})\circ d_2(u_i)),\text{ where } 2\leq i\leq b-1.\tag{1}
\end{equation}
From Proposition \ref{pro22} for $\alpha=u_i$,  (\ref{eq321}), (\ref{eq3319}), (\ref{eq3324}), (\ref{eq315}), (\ref{eq326}) and (\ref{eq328}), we conclude that
\begin{equation*}
\Phi_{u_i}(\sigma_{1}\circ pr(v_{i+1},\overline{v_1})\circ d_2(u_i))=\Phi_{u_i}((-1)^{i+1}2\pi_i),\text{ where } 2\leq i\leq b-1,
\end{equation*}
and the coefficient 2 comes from the multiplication in the divided power algebra $DF$. The map  $\Phi_{u_i}$ is 1-1, thus
\begin{equation}\label{eq2}
\sigma_1\circ pr(v_{i+1},\overline{v_1})\circ d_2(u_i)=(-1)^{i+1}2\pi_i,\text{ where } 2\leq i\leq b-1.\tag{2}
\end{equation}
From (\ref{eq1}), (\ref{eq2}) and (\ref{eq329}), it follows that
\begin{equation*}
f(u_i,v_1)(1)=(-1)^{i+1}2,\text{ where }  2\leq i\leq b-1.
\end{equation*}
The above, in combination with Definition \ref{def34} $ii)$ for $j=1$, yields (\ref{eq336}).
\end{proof}

\begin{proof}[\textbf{Proof of $\boldsymbol{\mathrm{(\ref{eq337})}}$}]From Definition \ref{def34} $i)$, (\ref{eq3113}), (\ref{eq3114}), Proposition \ref{pro37}(4),  (\ref{eq3323}) and (\ref{eq329}), it follows that 
\begin{equation*}
f(u_i,v_j)(e_l)=0\in\mathbb{Z},
\end{equation*}
where  $2\leq i\leq b-1,\; 2\leq j\leq b$  with  $j\neq i+1$ and  $1\leq l\leq b-1$.  The above, in combination with Definition \ref{def34} $ii)$, yields (\ref{eq337}).
\end{proof}

\begin{proof}[\textbf{Proof of $\boldsymbol{\mathrm{(\ref{eq338})}}$}]
From Definition \ref{def34} $i)$ for $j=i+1$,  (\ref{eq3113}), (\ref{eq3114}), Proposition \ref{pro37}(4) and (\ref{eq3323}), it follows that
\begin{equation}\label{eq3}
f(u_i,v_{i+1})(e_l)=\psi_i(\sigma_{i+1l+1}\circ pr(\overline{v_{i+1}},v_1)\circ d_2(u_i)),\tag{3}
\end{equation} 
where  $2\leq i\leq b-1$  and  $1\leq l\leq b-1$. From Proposition \ref{pro22} for $\alpha=u_i$,  (\ref{eq321}), (\ref{eq3319})  and (\ref{eq3324}), we conclude that
\begin{align}\label{eq4}
&\Phi_{u_i}(\sigma_{i+1l+1}\circ pr(\overline{v_{i+1}},v_1)\circ d_2(u_i))=
\tag{4}\\
\nonumber&=\sigma_{i+1l+1}(1^{(a)}\otimes1\otimes2\otimes\ldots\otimes(i-1)\otimes i^{(2)}\otimes (i+1)\otimes\ldots\otimes(b-1)),
\end{align}
where  $2\leq i\leq b-1$ and $ 1\leq l\leq b-1$.\\ 
$\bullet$ Let $l$=1 : From (\ref{eq4}), (\ref{eq3110}), (\ref{eq326}) and (\ref{eq328}), it follows that 
\begin{equation*}
\Phi_{u_i}(\sigma_{i+12}\circ pr(\overline{v_{i+1}},v_1)\circ d_2(u_i))=
\Phi_{u_i}((a+1)\pi_{i}),\text{ where }  2\leq i\leq b-1,	
\end{equation*}
and the coefficient $a+1$ comes from the multiplication in the divided power algebra $DF$. The map $\Phi_{u_i}$ is 1-1, thus \begin{equation}\label{eq5}
\sigma_{i+12}\circ pr(\overline{v_{i+1}},v_1) \circ d_2(u_i)=(a+1)\pi_{i},\text{ where }  2\leq i\leq b-1.\tag{5}
\end{equation}
$\bullet$ Let $l\neq 1$ : From (\ref{eq4}) and (\ref{eq3110}), it follows that 
\begin{equation*}
\Phi_{u_i}(\sigma_{i+1l+1}\circ pr(\overline{v_{i+1}},v_1)\circ d_2 (u_i))=d'_\mu(\mathrm{X}_{Q_{il}}), \text{ where } 2\leq i,l\leq b-1, \text{and}
\end{equation*}
\begin{equation*}
\mathrm{X}_{Q_{il}}=\begin{cases}
1^{(a)}l\otimes12\dots\widehat{l}\dots(i-1)i^{(2)}(i+1)\dots(b-1),&\text{ if } 2\leq l\leq i-1\\ 
1^{(a)}i\otimes12\dots(b-1),&\text{ if } l=i\\
1^{(a)}l\otimes12\dots(i-1)i^{(2)}(i+1)\dots\widehat{l}\dots(b-1),&\text{ if } i+1\leq l\leq b-1.
\end{cases}
\end{equation*}
In each case, $Q_{il}$ is a non-standard  tableau. Therefore, we apply the straightening law  {\cite[pp.235 - 236]{ABW}}, and deduce that for $2\leq i\leq b-1$, it holds 
\begin{equation*}
d'_\mu(\mathrm{X}_{Q_{il}})=\begin{cases}
-d'_\mu(\mathrm{X}_{P_i}), &\text{ if } 2\leq l\leq b-1 \text{ and } l\neq i\\
-2d'_\mu(\mathrm{X}_{P_i}), &\text{ if } l=i,
\end{cases}
\end{equation*}
where $\mathrm{X}_{P_i}$ is given by (\ref{eq326}). From the above, (\ref{eq328}) and the fact that $\Phi_{u_i}$
is an isomorphism of $\mathbb{Z}-$modules, we conclude that
\begin{align}\label{eq6}
&\text{for } 2\leq i\leq b-1, \text{we have}\tag{6}\\
&\nonumber\sigma_{i+1l+1}\circ pr(\overline{v_{i+1}},v_1)\circ d_2(u_i)=\begin{cases}
-\pi_i,&\text{if } 2\leq l\leq b-1 \text{ and } l\neq i\\
-2\pi_{i},&\text{if } l=i,\end{cases}
\end{align}
where $\pi_i$ is given by (\ref{eq327}). From (\ref{eq3}), (\ref{eq5}), (\ref{eq6}) and (\ref{eq329}), it follows that
\begin{equation*}
\text{for } 2\leq i\leq b-1, \; f(u_i,v_{i+1})(e_l)=\begin{cases}
a+1,&\text{ if } l=1\\
-1,&\text{ if } 2\leq l\leq b-1 \text{ and } l\neq i\\
-2,&\text{ if } l=i.\end{cases}
\end{equation*}
The above, in combination with Definition (\ref{def34}) $ii)$ for $j=i+1$, yields (\ref{eq338}).
\end{proof}

\begin{proof}[\textbf{Proof of $\boldsymbol{\mathrm{(\ref{eq339})}}$}]Definition \ref{def35} $i)$ for $j=1$, (\ref{eq318}), (\ref{eq3114}), Proposition \ref{pro37}(2) and  (\ref{eq3213}), imply that $f(w_i,v_1)(1)=0\in\mathbb{Z}^{b-2},\text{ where } 2\leq i\leq b-1$. The above, in combination with Definition \ref{def35} $ii)$ for $j=1$, yields (\ref{eq339}).
\end{proof}

\begin{proof}[\textbf{Proof of $\boldsymbol{\mathrm{(\ref{eq3310})}}$ and $\boldsymbol{\mathrm{(\ref{eq3314})}}$}]Similar to the proof of (\ref{eq337}).
\end{proof}

\begin{proof}[\textbf{Proof of $\boldsymbol{\mathrm{(\ref{eq3311})}}$}]From Definition \ref{def35} $i)$ for $j=i$,  (\ref{eq3113}), (\ref{eq3114}), Proposition \ref{pro37}(5) and (\ref{eq3323}), it follows that
\begin{equation}\label{eq7}
f(w_i,v_i)(e_l)=\theta_i(\sigma_{il+1}\circ pr(\overline{v_i},v_{i+1})\circ d_2(w_i)),\tag{7}
\end{equation}
where $2\leq i\leq b-1$ and  $1\leq l\leq b-1$. From Proposition \ref{pro22} for $\alpha=w_i$,  (\ref{eq321}), (\ref{eq3320})  and  (\ref{eq3324}),  we conclude that 
\begin{align*}
&\Phi_{w_i}(\sigma_{il+1}\circ pr(\overline{v_i},v_{i+1})\circ d_2(w_i))=\\
&=(-1)^{i+1}\sigma_{il+1}(1^{(a)}\otimes2\otimes\ldots\otimes(i-1)\otimes i^{(2)}\otimes i\otimes (i+1)\otimes\ldots\otimes(b-1)),
\end{align*}	
where $2\leq i\leq b-1$ and $1\leq l\leq b-1$.  Considering also (\ref{eq3110}), (\ref{eq3210}),
(\ref{eq3212}) and the fact that $\Phi_{w_i}$ is an isomorphism of $\mathbb{Z}-$modules, it follows that for $2\leq i\leq b-1$,
\begin{equation*}
\sigma_{il+1}\circ pr(\overline{v_i},v_{i+1})
\circ d_2(w_i)=\begin{cases} (-1)^{i+1}3{\pi_{il+1}},&\text{ if } 1\leq l\leq i-2\\
(-1)^{i+1}2{\pi_{ii}},&\text{ if }  l=i-1\\
(-1)^{i+1}{\pi_{ii}},&\text{ if } l=i\\
(-1)^{i+1}3{\pi_{il}},&\text{ if } i+1\leq l\leq b-1,\end{cases}
\end{equation*}
where $\pi_{ij}$ is given by (\ref{eq3211}). This, together with (\ref{eq7}) and (\ref{eq3213}), implies that for $2\leq i\leq b-1$, 
\begin{equation*}
f(w_i,v_i)(e_l)=\begin{cases} (-1)^{i+1}3{\epsilon_l},&\text{ if } 1\leq l\leq i-2\\
(-1)^{i+1}2{\epsilon_l},&\text{ if } l=i-1\\
(-1)^{i+1}{\epsilon_{l-1}},&\text{ if } l=i\\
(-1)^{i+1}3{\epsilon_{l-1}},&\text{ if } i+1\leq l\leq b-1.\end{cases}
\end{equation*}
The above, in combination with Definition \ref{def35} $ii)$ for $j=i$, yields (\ref{eq3311}).
\end{proof}

\begin{proof}[\textbf{Proof of $\boldsymbol{\mathrm{(\ref{eq3312}),(\ref{eq3315})}}$ and $\boldsymbol{\mathrm{(\ref{eq3316})}}$}] Similar to the proof of (\ref{eq3311}).
\end{proof}

\begin{proof}[\textbf{Proof of $\boldsymbol{\mathrm{(\ref{eq3313})}}$}]
Similar to the proof of (\ref{eq339}).
\end{proof}

\section{The invariant factors of matrix $M$ for $b\geq7$}In this section, we determine the invariant factors of the matrix $M=M(a,b)$, where $b\geq 7$ and $a+1>b-1$, by performing row and column operations. We denote by  $D_n=D_n(a,b)$ the $n-th$ operation applied to $M$ and $M_n=M_n(a,b):=D_nM_{n-1}$, where $M_0:=M$. From (\ref{eq335}) we obtain the following

\begin{proposition}\label{pro41}
The matrix $M_n$  is of the form
\begin{equation*}
M_n=\begin{bmatrix}
A_n(u_2,v_1) & A_n(u_2,v_2) & \cdots & A_n(u_2,v_b)\\
\vdots & \vdots & & \vdots\\
A_n(u_{b-1},v_1) & A_n(u_{b-1},v_2) & \cdots & A_n(u_{b-1},v_b)\\
E_n(2,v_1) & E_n(2,v_2) & \cdots & E_n(2,v_b)\\
\vdots & \vdots & & \vdots\\
E_n(b-1,v_1) & E_n(b-1,v_2) & \cdots & E_n(b-1,v_b)
\end{bmatrix}, 
\end{equation*} 
where 
\begin{equation}\label{eq41a}
A_n(u_i,v_j):=D_nA_{n-1}(u_i,v_j), \text{where } A_0(u_i,v_j):= A(u_i,v_j), \;and\tag{4.1.a}
\end{equation}
\begin{equation*}
E_n(i,v_j):=D_nE_{n-1}(i,v_j),
\text{where } E_0(i,v_j):=E(i,v_j),
\end{equation*}
consequently
\begin{equation*}
E_n(i,v_j)=\begin{bmatrix}
B_n(w_i,v_j)\\
C_n(w_{ii+1},v_j)\\
\vdots\\
C_n(w_{ib-1},v_j)
\end{bmatrix}, 
\text{ if } i\neq b-1, \text{and } E_n(b-1,v_j)=B_n(w_{b-1},v_j),
\end{equation*}
where
\begin{equation}\label{eq41b}
B_n(w_i,v_j):=D_nB_{n-1}(w_i,v_j), \text{where } B_0(w_i,v_j):=B(w_i,v_j),\text{and}\tag{4.1.b}
\end{equation}
\begin{equation}\label{eq41c}
C_n(w_{ij},v_h):=D_nC_{n-1}(w_{ij},v_h), \text{where } C_0(w_{ij},v_h):=
C(w_{ij},v_h).\tag{4.1.c}
\end{equation}
\end{proposition}
Now, let
\begin{equation*}
A_n(u_i):=\left[A_n(u_i,v_1)\quad A_n(u_i,v_2)\quad\ldots\quad A_n(u_i,v_b)\right],\text{where } 2\leq i\leq b-1,
\end{equation*}
be the block of the matrix $M_n$ corresponding to $u_i$. Similarly, we define $B_n(w_i)$, where $2\leq i\leq b-1$, and $C_n(w_{ij})$, where $2\leq i<j\leq b-1$. In addition, 
\begin{equation*}
M_n(v_j):=\left[A_n(u_2,v_j)\quad\ldots\quad A_n(u_{b-1},v_j)\quad E_n(2,v_j)\quad\ldots\quad E_n(b-1,v_j)\right]^{tr}, 
\end{equation*}
where $1\leq j\leq b$, is the block of $v_j$ in the  matrix $M_n$.

We consider matrix $M$ and focus on the block $B(w_i)$, \text{where }$2\leq i\leq b-1$, given by (\ref{eq339}) - (\ref{eq3312}).
\begin{remark}\label{rem42}
We observe that	for $2\leq i\leq b-1$ and $1\leq l\leq b-2$ with $l\neq i-1$, $\Gamma_lB(w_i)$ contains exactly two non-zero entries, each equal to $(-1)^{i+1}3$, located 
\begin{itemize}
\item if $1\leq l\leq i-2$, at $\Sigma_l(M(v_i))$ and $\Sigma_l(M(v_{i+1}))$, respectively
\item if $i\leq l\leq b-2$, at $\Sigma_{l+1}(M(v_i))$ and $\Sigma_{l+1}(M(v_{i+1}))$, respectively. 
\end{itemize}
\end{remark} 

Our goal is that for $2\leq i\leq b-1$ and $1\leq l\leq b-2$ with $l\neq i-1$, the $l-th$ row of the block corresponding to $w_i$ to contain a unique non-zero entry equal to $(-1)^{i+1}3$, located, if $1\leq l\leq i-2$, at the $l-th$ column of the block of $v_i$, and if $i\leq l\leq b-2$,  at the $(l+1)-th$ column of the block of  $v_{i+1}$.
This will be achieved by applying operations $D_1$ and $D_2$ to matrix $M$.
\subsection{Operation $\boldsymbol{D_1}$}We focus on the block $B(w_i)$, where $i=2$ or $i=b-1$, given by (\ref{eq339}) - (\ref{eq3312}). We will apply the appropriate column operations to matrix $M$ to vanish $-3$  and $(-1)^{b}3$, located at $B(w_2,v_2)$ and $B(w_{b-1},v_b)$, respectively.
\begin{definition}\label{def43}
We define operation $D_1$ as
\begin{equation*}
D_1:=\{D_1(v_2,\tau),\; D_1(v_b,r):\; 3\leq\tau\leq b-1,\; 1\leq r\leq b-3\},
\end{equation*}
where
\begin{align*}
&D_1(v_2,\tau):\Sigma_\tau\left(M(v_2)\right) +\sum_{j=3}^{\tau}(-1)^{j}\Sigma_\tau\left(M(v_j)\right),\;  3\leq\tau\leq b-1, \text{ and }\\
&D_1(v_b,r):\Sigma_r\left(M(v_b)\right) +\sum_{j=r+2}^{b-1}(-1)^{j+b}\Sigma_r\left(M(v_j)\right), 1\leq r\leq b-3.
\end{align*}
\end{definition}
\begin{remark}\label{rem44}
From Definition \ref{def43}, we conclude that applying operation $D_1$ to matrix $M$ will cause changes only in the blocks of $v_2$ and $v_b$.
\end{remark}
Now, we obtain matrix $M_1:=D_1M$, given by  Proposition \ref{pro41} for $n=1$. From (\ref{eq41a}) for $n=1$ and Remark \ref{rem44}, it follows that
\begin{equation}\label{eq411}
A_1(u_i,v_j)=A(u_i,v_j),\text{ where } 2\leq i\leq b-1 \text{ and } j=1 \text{ or }3\leq j\leq b-1,
\end{equation}
and $A(u_i,v_j)$ is given by (\ref{eq336}), if $j$=1, by (\ref{eq337}), if $3\leq j\leq b-1$ with $j\neq i+1$, and by (\ref{eq338}), if $j=i+1\neq b$.
 
From (\ref{eq41a}) for $n=1$ and Definition \ref{def43}, we have that for $2\leq i\leq b-1$,
\begin{equation*}
\Sigma_\tau A_1(u_i, v_2)=\begin{cases}
\Sigma_\tau A(u_i,v_2), \text{ if } \tau=1 \text{ or } \tau=2\\\\
\Sigma_\tau A(u_i,v_2)+
\sum\limits_{j=3}^{\tau}(-1)^{j}\Sigma_\tau A(u_i,v_j), \text { if } 3\leq\tau\leq b-1\end{cases}
\end{equation*}
and
\begin{equation*}
\Sigma_rA_1(u_i,v_b)=\begin{cases}
\Sigma_rA(u_i,v_b)+\sum\limits_{j=r+2}^{b-1} (-1)^{j+b}\Sigma_rA(u_i,v_j), \text{ if } 1\leq r\leq b-3\\\\
\Sigma_rA(u_i,v_b), \text { if } r=b-2 \text{ or } r=b-1.\end{cases}
\end{equation*}
Considering (\ref{eq337}) and (\ref{eq338}) as well, we obtain the following relations:
\begin{equation}\label{eq412}
A_1(u_i,v_2)=(-1)^i\left[O_{1\times i}
\quad1\quad\dots\quad1\right]\in\mathbb{Z}^{1\times(b-1)}, \text{ where }  2\leq i\leq b-1,
\end{equation} 
\begin{equation}\label{eq413}
A_1(u_i,v_b)=(-1)^{i+1+b}\left[a+1\quad-1\quad\dots\quad-1\quad O_{1\times(b-i)}\right]\in\mathbb{Z}^{1\times (b-1)}, 
\end{equation} 
where  $2\leq i\leq b-2$, and
\begin{equation}\label{eq414}
A_1(u_{b-1},v_b)=A(u_{b-1},v_b),\text{ where } A(u_{b-1},v_b) \text{ is given by }(\ref{eq338}).
\end{equation}

Similarly, from (\ref{eq41b}) and (\ref{eq41c}) for $n=1$, Remark \ref{rem44}, Definition \ref{def43} , (\ref{eq3310})- (\ref{eq3312}) and (\ref{eq3314}) - (\ref{eq3316}), we obtain the following relations :
\begin{equation} \label{eq415}
B_1(w_i,v_j)=B(w_i,v_j),\text{ where } 2\leq i\leq b-1\text{ and }j=1\text{ or }3\leq j\leq b-1,
\end{equation} 
and $B(w_i,v_j)$ is given by (\ref{eq339}), if $j$=1, by (\ref{eq3310}), if $3\leq j\leq b-1$ with $j\neq i$ and $j\neq i+1$, by (\ref{eq3311}), if $j=i\neq2$, and by (\ref{eq3312}), if $j=i+1\neq b$,
\begin{equation}\label{eq416}
B_1(w_2,v_2)=-2E_1-E_{1,2}\in\mathbb{Z}^{(b-2)\times (b-1)},
\end{equation}
\begin{equation}\label{eq417}
B_1(w_i,v_2)=-E_{i-1,i}\in\mathbb{Z}^{(b-2)\times(b-1)},\text{ where } 3\leq i\leq b-1,
\end{equation}
\begin{equation}\label{eq418}
B_1(w_i,v_b)=(-1)^bE_{i-1}\in\mathbb{Z}^{(b-2)\times(b-1)},\text{ where }2\leq i\leq b-2, 
\end{equation}
\begin{equation}\label{eq419}
B_1(w_{b-1},v_b)=(-1)^b(E_{b-2}+2E_{b-2,b-1})\in\mathbb{Z}^{(b-2)\times(b-1)},
\end{equation}
\begin{equation}\label{eq4110}
C_1(w_{ij},v_h)=C(w_{ij},v_h),\text{where } 2\leq i<j\leq b-1 \text{ and }h=1\text{ or }3\leq h\leq b-1,
\end{equation}
and $C(w_{ij},v_h)$ is given by (\ref{eq3313}), if $h$=1, by (\ref{eq3314}), if $3\leq h\leq b-1$ with $h\neq i$ and $h\neq j+1$, by (\ref{eq3315}), if $h=i\neq2$, and by (\ref{eq3316}), if $h=j+1\neq b$,
\begin{equation}\label{eq4111}
C_1(w_{2j},v_2)=(-1)^{j+1}\begin{bmatrix}
2I_{j-2}& & &\\
&1&1 &\\
& & & O_{(b-2)-(j-1)}
\end{bmatrix},
\text{ where } 2<j\leq b-1,
\end{equation}
\begin{equation}\label{eq4112}
C_1(w_{ij},v_2)=(-1)^{i+j+1}\begin{bmatrix}
O_{i-1}& & &\\
&2I_{(j-2)-(i-1)}& &\\
& & 1&1 &\\
& & & &O_{(b-2)-(j-1)} 
\end{bmatrix},
\end{equation}
where  $3\leq i<j\leq b-1$,
\begin{equation}\label{eq4113}
C_1(w_{ij},v_b)=(-1)^{i+j+b}\begin{bmatrix}
O_{i-2}& & & &\\
& 1& 1& &\\
&  &  & 2I_{(j-2)-(i-1)} &\\
& & & & O_{(b-2)-(j-2)}
\end{bmatrix},
\end{equation}
where  $2\leq i<j\leq b-2$, and
\begin{equation}\label{eq4114}
C_1(w_{ib-1},v_b)=(-1)^{i+1}\begin{bmatrix}
O_{i-2}& & &\\
& 1& 1&\\
&  &  & 2I_{(b-2)-(i-1)}
\end{bmatrix},
\text{ where } 2\leq i<b-1.
\end{equation}
\textbf{Results of $\boldsymbol{D_1}$.} In matrix $M_1$, (\ref{eq415}) - (\ref{eq419}) lead to the following
\begin{corollary}\label{cor45}
For $i=2$ or $i=b-1$ and $1\leq l\leq b-2$ with $l\neq i-1$, $\Gamma_lB_1(w_i)$ contains a unique non-zero entry equal to $(-1)^{i+1}3$, located, if  $i=2$, at  $\Sigma_{l+1}M_1(v_{i+1})$, and if $i=b-1$, at  $\Sigma_lM_1(v_i)$.
\end{corollary}
From Corollary \ref{cor45}, we conclude that by applying operation $D_1$ to matrix $M$, we achieve the goal introduced after Remark \ref{rem42}, for $i=2$ or $i=b-1$ and $1\leq l\leq b-2$ with $l\neq i-1$.

From (\ref{eq4110}) - (\ref{eq4114}), we obtain two more useful results on $M_1$. Specifically,
\begin{corollary}\label{cor46}\text{ }
\begin{enumerate}
\item For $2<j\leq b-2$ and  $j\leq l\leq b-2$, $\Gamma_lC_1(w_{2j})$ contains a unique non-zero entry equal to $-2$, located at $\Sigma_{l+1}M_1(v_{j+1})$.
\item For $3\leq i<b-1$ and $1\leq l\leq i-2$,  $\Gamma_lC_1(w_{ib-1})$ contains a unique non-zero entry equal to $(-1)^{b}2$, located at $\Sigma_l M_1(v_i)$.
\end{enumerate}
\end{corollary}

\subsection{Operation $\boldsymbol{D_2}$}From (\ref{eq339}) - (\ref{eq3312}) and (\ref{eq415}) - (\ref{eq419}), it follows that for  $3\leq i\leq b-2$ and  $1\leq l\leq b-2$ with $l\neq i-1$, it holds $\Gamma_lB_1(w_i)=\Gamma_lB(w_i)$. We know that $\Gamma_lB(w_i)$ satisfies Remark \ref{rem42}, hence $M_1$ has the following property:
\begin{corollary}\label{cor47}
For $3\leq i\leq b-2$ and $1\leq l\leq b-2$ with $l\neq i-1$, $\Gamma_lB_1(w_i)$ contains exactly two non-zero entries, each equal to $(-1)^{i+1}3$, located 
\begin{itemize}
\item if\; $1\leq l\leq i-2$, at $\Sigma_l(M_1(v_i))$  and  $\Sigma_l(M_1(v_{i+1}))$, respectively
\item if\; $i\leq l\leq b-2$, at $\Sigma_{l+1}(M_1(v_i))$ and  $\Sigma_{l+1}(M_1(v_{i+1}))$, respectively.
\end{itemize}
\end{corollary}

Now, we will apply the appropriate row operations to $M_1$
to vanish $(-1)^{i+1}3$, where $3\leq i\leq b-2$, located at $\Gamma_lB_1(w_i,v_{i+1})$, where $1\leq l\leq i-2$, and at $\Gamma_lB_1(w_i,v_i)$,
where $i\leq l\leq b-2$. Thus, we need the following

\begin{proposition}\label{pro48}
We consider matrix $M_1$. For $3\leq i\leq b-2$, the following hold :
\begin{enumerate}
\item If\; $1\leq l\leq i-2$, the sum of rows $\sum\limits_{h=i+1}^{b-1}\Gamma_lB_1(w_h)$ contains a unique non-zero entry equal to $(-1)^i3$, located at the $l-th$ column of the block of $v_{i+1}$.
\item If\; $i\leq l\leq b-2$, the sum of rows $\sum\limits_{h=2}^{i-1}\Gamma_lB_1(w_h)$ contains a unique non-zero entry equal to $(-1)^i3$, located at the $(l+1)-th$ column of the block of $v_i$.
\end{enumerate}
\end{proposition}
\begin{proof}(1) We proceed by induction on $\tau=(b-1)-i$  $(1\leq\tau\leq b-4)$.
Let $\tau=1$. Then $i=b-2$, hence for $1\leq l\leq i-2=b-4$, we have
\begin{equation*}
\sum\limits_{h=i+1}^{b-1}\Gamma_lB_1(w_h)=\Gamma_lB_1(w_{b-1}).
\end{equation*}
Furthermore, according to Corollary \ref{cor45}, $\Gamma_lB_1(w_{b-1})$ contains a unique non-zero entry equal to $(-1)^b3$, located at the $l-th$ column of the block of $v_{b-1}$, $\Sigma_lM_1\left(v_{b-1}\right)$. Therefore, Proposition \ref{pro48}(1) holds for $i=b-2$, equivalently, for $\tau=1$.
Let $\tau =(b-1)-i$, where $1\leq\tau\leq b-5$ (equivalently, $i=(b-1)-\tau$, where $4\leq i\leq b-2$). We assume that (induction hypothesis) Proposition \ref{pro48}(1) holds for  $\tau$. We will show that Proposition \ref{pro48}(1) holds for $\tau+1$. Let $1\leq l\leq(i-1)-2=i-3$. Then
\begin{equation*}
\sum\limits_{h=(i-1)+1}^{b-1}\Gamma_lB_1(w_h)=\sum\limits_{h=i}^{b-1}\Gamma_lB_1(w_h)=\Gamma_lB_1(w_i)+\sum\limits_{h=i+1}^{b-1}\Gamma_lB_1(w_h),
\end{equation*}
which, in combination with Corollary \ref{cor47} and the induction hypothesis for $i$, implies that Proposition \ref{pro48}(1) holds for $i-1$, equivalently, for $\tau+1$. Now the proof is complete.
(2) Similarly, by induction on $i$ ($3\leq i\leq b-2$).
\end{proof}

\begin{definition}\label{def49}
We define operation $D_2$ as 
\begin{equation*}
D_2:=\{D_2(w_i,l):\; 3\leq i\leq b-2,\; 1\leq l\leq b-2 \text{ and } l\neq i-1\},
\end{equation*}
where for $3\leq i\leq b-2$,
\begin{equation*}
D_2(w_i,l):\begin{cases}
\Gamma_lB_1(w_i)+\sum\limits_{h=i+1}^{b-1}\Gamma_lB_1(w_h),\text{ where } 1\leq l\leq i-2\\
\Gamma_lB_1(w_i)+\sum\limits_{h=2}^{i-1}\Gamma_lB_1(w_h),\text{ where } i\leq l\leq b-2.
\end{cases}
\end{equation*}
\end{definition}

Definition \ref{def49}, Corollary \ref{cor47} and Proposition \ref{pro48}, yield the following 

\begin{corollary}\label{cor410}
The unique change caused by applying operation $D_2$ to matrix $M_1$ is the nulling of $(-1)^{i+1}3$, where $3\leq i\leq b-2$, located at $\Gamma_lB_1(w_i,v_{i+1})$, where  $1\leq l\leq i-2$, and at $\Gamma_lB_1(w_i, v_i)$, where $i\leq l\leq b-2$.
\end{corollary}

Now, we obtain matrix $M_2:=D_2M_1=D_2D_1M$, given by Proposition \ref{pro41} for $n=2$. From (\ref{eq41a}) for $n=2$ and Corollary \ref{cor410}, it follows that
\begin{equation}\label{eq421}
A_2(u_i,v_j)=A_1(u_i,v_j), \text{where }  2\leq i\leq b-1 \text{ and } 1\leq j\leq b, 
\end{equation}
and $A_1(u_i,v_j)$ is given by (\ref{eq411}) - (\ref{eq414}).

From (\ref{eq41b}) for $n=2$, Corollary \ref{cor410}, and (\ref{eq415}), we obtain the following :
\begin{equation}\label{eq422}
B_2(w_i,v_j)=B(w_i,v_j),
\end{equation} 
where $2\leq i\leq b-1$ and $j=1$ or $3\leq j\leq b-1$ with $j\neq i$ and $j\neq i+1$ or $j=i=b-1$ or $j=i+1=3$, and $B(w_i,v_j)$ is given by (\ref{eq339}), (\ref{eq3310}), (\ref{eq3311}) and (\ref{eq3312}), respectively, and
\begin{equation}\label{eq423}
B_2(w_i,v_j)=B_1(w_i,v_j),\text{ where } 2\leq i\leq b-1 \text{ and } j=2 \text{ or } j=b,
\end{equation}
and $B_1(w_i,v_j)$ is given by (\ref{eq416}) - (\ref{eq417}) and (\ref{eq418}) - (\ref{eq419}), respectively.

From (\ref{eq41b}) for $n=2$, (\ref{eq415}), (\ref{eq3311}), (\ref{eq3312}) and Corollary \ref{cor410}, we conclude that 
\begin{equation}\label{eq424}
B_2(w_i,v_i)=(-1)^{i+1}\begin{bmatrix}
3I_{i-2}& & &\\
& 2& 1&\\
&  &  & O_{(b-2)-(i-1)} 
\end{bmatrix},
\text{ where } 3\leq i\leq b-2,\text{ and}
\end{equation}
\begin{equation}\label{eq425}
B_2(w_i,v_{i+1})=(-1)^{i+1}\begin{bmatrix}
O_{i-2}& & &\\
& 1& 2&\\
&  &  & 3I_{(b-2)-(i-1)}
\end{bmatrix},
\text{ where } 3\leq i\leq b-2.
\end{equation}

From (\ref{eq41c}) for $n=2$ and Corollary \ref{cor410}, it follows that
\begin{equation}\label{eq426}
C_2 (w_{ij},v_h)=C_1(w_{ij},v_h),\text{ where } 2\leq i<j\leq b-1 \text{ and }  1\leq h\leq b, 
\end{equation} 
and $C_1(w_{ij}, v_h)$ is given by (\ref{eq4110}) - (\ref{eq4114}).\\
\textbf{Result of $\boldsymbol{D_2}$.} According to (\ref{eq422}) - (\ref{eq425}), matrix  $M_2$ has the following property:
\begin{proposition}\label{pro411}
For $2\leq i\leq b-1$ and  $1\leq l\leq b-2$ with $l\neq i-1$,  $\Gamma_lB_2(w_i)$ contains a unique non-zero entry equal to $(-1)^{i+1}3$, located 
\begin{itemize}
\item if $1\leq l\leq i-2$, at $\Sigma_l M_2(v_i)$
\item if $i\leq l\leq b-2$, at $\Sigma_{l+1} M_2(v_{i+1})$.
\end{itemize}
\end{proposition}
Proposition \ref{pro411} verifies that the goal introduced after Remark \ref{rem42} is achieved by applying operations $D_1$ and $D_2$ to matrix $M$.

\subsection{Operation $\boldsymbol{D_3}$} We consider matrix $M_2$ and focus on the block  $C_2(w_{ij})$, where $2\leq i<j\leq b-1$,
given by (\ref{eq426}) and (\ref{eq4110}) - (\ref{eq4114}). 
\begin{remark}\label{rem412}In matrix $M_2$, we notice the following:
\begin{enumerate}
\item For $2<j\leq b-2$ and  $1\leq l\leq j-2$,\\  $\Gamma_lC_2(w_{2b-1},v_h)=(-1)^{j+b+1}\Gamma_lC_2(w_{2j},v_h)\neq O_{1\times(b-1)}$, if $h=2$ or $h=b$,  while the remaining blocks of $\Gamma_lC_2(w_{2b-1})$ are zero.
\item For $3\leq i<j\leq b-2$ and $1\leq l\leq j-2$,\\ $\Gamma_lC_2(w_{ib-1},v_h)=(-1)^{j+b+1}\Gamma_lC_2(w_{ij},v_h)\neq O_{1\times(b-1)}$, if $h=2$ and $i\leq l$ or $h=i$ or $h=b$ and $i-1\leq l$, while the remaining blocks of $\Gamma_lC_2(w_{ib-1})$ are zero.
\item For $3\leq i<j\leq b-2$ and  $j-1\leq l\leq b-2$,\\
$\Gamma_lC_2(w_{2j},v_h)=(-1)^i\Gamma_lC_2(w_{ij},v_h)\neq O_{1\times(b-1)}$, if $h=2$ and $l=j-1$ or $h=j+1$, while the remaining blocks of $\Gamma_lC_2(w_{2j})$ are zero.
\item For $3\leq i<b-1$ and $i\leq l\leq b-2$, \\ $\Gamma_lC_2(w_{2b-1},v_h)=(-1)^i\Gamma_lC_2(w_{ib-1},v_h)\neq O_{1\times(b-1)}$, if $h=2$ or $h=b$, while the remaining blocks of $\Gamma_lC_2(w_{2b-1})$ are zero.
\end{enumerate} 
\end{remark} 
Considering also block $B_2(w_{b-1})$, given by (\ref{eq422}) - (\ref{eq423}), we  calculate that
\begin{align}\label{eq431}
&\Gamma_lC_2(w_{2b-1})+\Gamma_{l+1}C_2(w_{2b-1})+(-1)^{l+b+1}2\Gamma_lC_2(w_{2l+1})+\\
&\nonumber+(-1)^{l+1}2\Gamma_{l+1}C_2(w_{l+2b-1})=\Gamma_{b-3}C_2(w_{2b-1})+2\Gamma_{b-2}C_2(w_{2b-1})+\\
&\nonumber +2\Gamma_{b-3}C_2(w_{2b-2})+\Gamma_{b-2}C_2(w_{2b-2})+(-1)^b2\Gamma_{b-2}B_2(w_{b-1})=\\
&\nonumber=O_{1\times[1+(b-1)^2]}, \text{ where }2\leq l\leq b-4.
\end{align}
 
Remark \ref{rem412} and (\ref{eq431}) lead to the following
\begin{definition}\label{def413}
We define operation $D_3$ as
\begin{equation*}
\begin{array}{ll}
D_3:=&\{D_3(w_{2j},l):\ 2<j\leq b-2,\ 1\leq l\leq j-2\}\cup\{D_3(w_{ij},l):\\
&3\leq i<j\leq b-2,\ 1\leq l\leq b-2\}\cup
\{D_3(w_{ib-1},l):\ 3\leq i<b-1,\\
&i\leq l\leq b-2\}\cup
\{D_3(w_{2b-1},l):\ 2\leq l\leq b-3\},
\end{array}
\end{equation*}
where
\begin{align*}
&\bullet\ D_3(w_{2j},l): \Gamma_lC_2(w_{2j})+(-1)^{j+b}\Gamma_lC_2(w_{2b-1}), \text{where } 2<j\leq b-2 \text{ and }\\
&\hspace*{2,27cm}1\leq l\leq j-2\\
&\bullet\
\text{for } 3\leq i<j\leq b-2,\\
&\hspace*{1,2cm}D_3(w_{ij},l):\begin{cases}
\Gamma_lC_2(w_{ij})+(-1)^{j+b}\Gamma_lC_2(w_{ib-1}),\text{where } 1\leq l\leq j-2\\
\Gamma_lC_2(w_{ij})+(-1)^{i+1}\Gamma_lC_2(w_{2j}),\text{where } j-1\leq l\leq b-2
\end{cases}\\
&\bullet\ 
D_3(w_{ib-1},l): \Gamma_lC_2(w_{ib-1})+(-1)^{i+1} \Gamma_lC_2(w_{2b-1}),\text{where } 3\leq i<b-1 \text{ and }\\
&\hspace*{2,6cm}i\leq l\leq b-2\\
&\bullet\
D_3(w_{2b-1},l):\Gamma_lC_2(w_{2b-1})+ \Gamma_{l+1}C_2(w_{2b-1})+(-1)^{l+b+1}2\Gamma_lC_2(w_{2l+1})+\\
&\hspace*{2,55cm}+(-1)^{l+1}2\Gamma_{l+1}C_2(w_{l+2b-1}), \text{where }2\leq l\leq b-4\\
&\bullet\ 
D_3(w_{2b-1},b-3):\Gamma_{b-3}C_2(w_{2b-1})+2 \Gamma_{b-2}C_2(w_{2b-1})+2\Gamma_{b-3}C_2(w_{2b-2})+\\ 
&\hspace*{3,15cm}+\Gamma_{b-2}C_2(w_{2b-2})+(-1)^b2\Gamma_{b-2}B_2(w_{b-1}).
\end{align*}
\end{definition}

Remark \ref{rem412}, (\ref{eq431}) and Definition \ref{def413} yield the following
\begin{corollary}\label{cor414}
The unique change caused by applying operation $D_3$ to matrix $M_2$ is the nulling of the following non-zero rows :
\begin{itemize}
\item$\Gamma_lC_2(w_{ij},v_2)$, \text{where} $2=i<j\leq b-2$ \text{and} $1\leq l\leq j-2$ \text{ or } $2=i<j=b-1$ \text{and} $2\leq l\leq b-3$ \text{ or } $3\leq i<j\leq b-1$ \text{and} $i\leq l\leq j-1$
\item$\Gamma_lC_2(w_{ij},v_i)$, \text{where} $3\leq i<j\leq b-2$ \text{and} $1\leq l\leq j-2$
\item$\Gamma_lC_2(w_{ij},v_{j+1})$, \text{where} $3\leq i<j\leq b-2$ \text{and} $j-1\leq l\leq b-2$
\item$\Gamma_lC_2(w_{ij},v_b)$, \text{where} $2\leq i<j\leq b-2$ \text{and} $i-1\leq l\leq j-2$ \text{ or } $2=i<j=b-1$ \text{and} $2\leq l\leq b-3$ \text{ or } $3\leq i<j=b-1$ \text{and} $i\leq l\leq b-2$.
\end{itemize}
\end{corollary}
We note that from (\ref{eq426}), (\ref{eq4110}), (\ref{eq3313}) and (\ref{eq3314}), we conclude that $\Gamma_lC_2(w_{2b-1},v_h)$ is zero, where $h=1$ or $3\leq h\leq b-1$  and $2\leq l\leq b-3$.

Now, we obtain  matrix $M_3:=D_3M_2=D_3D_2D_1M$, given by Proposition \ref{pro41} for $n=3$. From (\ref{eq41a}) and (\ref{eq41b}) for $n=3$, Corollary \ref{cor414} and (\ref{eq421}), it follows that
\begin{equation}\label{eq432}
A_3(u_i, v_j)=A_1(u_i,v_j), \text{ where } 2\leq i\leq b-1 \text{ and } 1\leq j\leq b,
\end{equation}
and $A_1(u_i,v_j)$ is given by (\ref{eq411}) - (\ref{eq414}), and
\begin{equation}\label{eq433}
B_3(w_i,v_j)=B_2(w_i,v_j), \text{where } 2\leq i \leq b-1 \text{ and } 1\leq j\leq b,	
\end{equation}
and $B_2(w_i,v_j)$ is given by (\ref{eq422}) - (\ref{eq425}).

From (\ref{eq41c}) for $n=3$, (\ref{eq426}), (\ref{eq4110}) - (\ref{eq4114}), (\ref{eq3313}) - (\ref{eq3316}) and Corollary \ref{cor414}, we obtain the following relations :
\begin{equation}\label{eq434}
C_3(w_{ij},v_h)=C(w_{ij}, v_h),
\end{equation}
where $2\leq i<j\leq b-1$ and $h=1$ or $2\leq i<j\leq b-1$ and $3\leq h\leq b-1$ with $h\neq i$ and $h\neq j+1$ or $3\leq i<j=b-1$ and $h=i$ or $2=i<j\leq b-2$ and $h=j+1$, and $C(w_{ij},v_h)$ is given by (\ref{eq3313}), (\ref{eq3314}), (\ref{eq3315}) and (\ref{eq3316}), respectively,
\begin{equation}\label{eq435}
C_3(w_{2j},v_2)=(-1)^{j+1}(E_{j-1}+E_{j-1,j})\in\mathbb{Z}^{(b-2)\times (b-1)}
,\text{ where } 2<j\leq b-2,
\end{equation}
\begin{equation}\label{eq436}
C_3(w_{2b-1},v_2)=(-1)^b(2E_1+E_{b-2}+E_{b-2,b-1})\in\mathbb{Z}^{(b-2)\times(b-1)},
\end{equation}
\begin{equation}\label{eq437}
C_3(w_{ij},v_2)=O_{(b-2)\times(b-1)}, \text{ where } 3\leq i<j\leq b-1, 
\end{equation}
\begin{equation}\label{eq438}
C_3(w_{ij},v_i)=(-1)^{j+1}
\begin{bmatrix}
O_{j-2}& & &\\
& 1& 1&\\
&  &  & 2I_{(b-2)-(j-1)}
\end{bmatrix},
\text{where } 3\leq i<j\leq b-2,
\end{equation}
\begin{equation}\label{eq439}
C_3(w_{ij},v_{j+1})=(-1)^{i+1}
\begin{bmatrix}
2I_{i-2}& & &&\\
& 1& 1&&\\
&  & &2I_{(j-2)-(i-1)} &\\
&&&&O_{(b-2)-(j-2)}
\end{bmatrix},
\end{equation}
where $3\leq i<j\leq b-2$,
\begin{equation}\label{eq4310}
C_3(w_{ij},v_b)=O_{(b-2)\times(b-1)}, \text{ where } 2\leq i<j\leq b-2,  
\end{equation}
\begin{equation}\label{eq4311}
C_3(w_{2b-1},v_b)=-E_1-E_{1,2}-2E_{b-2,b-1}\in\mathbb{Z}^{(b-2)\times(b-1)},\text{ and }
\end{equation}
\begin{equation}\label{eq4312}
 C_3(w_{ib-1},v_b)=(-1)^{i+1}(E_{i-1}+E_{i-1,i})\in\mathbb{Z}^{(b-2)\times(b-1)},\text{ where } 3\leq i<b-1.
\end{equation}
\textbf{Results of $\boldsymbol{D_{3}}$.}
From (\ref{eq433}) and Proposition \ref{pro411}, (\ref{eq434}) - (\ref{eq4312}), and  (\ref{eq3313}) - (\ref{eq3316}), we conclude that matrix $M_3$ has the following properties: 
\begin{corollary}\label{cor415}
Let $2\leq i\leq b-1$ and  $1\leq l\leq b-2$ with $l\neq i-1$. Then
\begin{enumerate}
\item $\Gamma_lB_3(w_i)$ contains a unique non-zero entry equal to $(-1)^{i+1}3$
\item there exists at least one row of the matrix $M_3$ in the blocks of $C_3(w_{rs})$, let us call it $\Gamma(i,l)$, which contains a unique non-zero entry equal to $\pm 2$
\item both the above entries $\left((-1)^{i+1}3\text{ and }\pm2\right)$ are located, if $1\leq l\leq i-2$, at $\Sigma_lM_3(v_i)$, and if $i\leq l\leq b-2$, at $\Sigma_{l+1}M_3(v_{i+1})$.
\end{enumerate}
\end{corollary}

\subsection{Operation $\boldsymbol{D_4}$}
We consider  matrix $M_3$, which satisfies Corollary \ref{cor415}, and define operation $D_4$ as follows:
\begin{definition}\label{def416}
Let $2\leq i\leq b-1$ and $1\leq l\leq b-2$ with $l\neq i-1$.\\
\textbf{Step $\boldsymbol{1.}$} We add or subtract an appropriate multiple of $\Gamma(i,l)$ from $\Gamma_lB_3(w_i)$ so that the $l-th$ row of the block of  $w_i$  contains a unique non-zero entry equal to $1$, located 
\begin{itemize}
\item if $1\leq l\leq i-2$, at the $l-th$ column of the block of  $v_i$
\item  if $i\leq l\leq b-2$, at the $(l+1)-th$ column of the block of $v_{i+1}$.
\end{itemize}
\textbf{Step $\boldsymbol{2.}$} We consider the rows of the matrix that each contain a non-zero entry in the column in question, excluding the $l-th$ row of the block of $w_i$. We add or subtract an appropriate multiple of the  $l-th$ row of the block of $w_i$ from each of the above rows, to null all the non-zero entries of the column in question, excluding $1$, described in Step $1$.
\end{definition}

\begin{remark}\label{rem417}
According to Definition \ref{def416}, the only impact of operation $D_4$ on matrix $M_3$ is the following : for $2\leq i\leq b-1$ and  $1\leq l\leq b-2$ with $l\neq i-1$, the unique non-zero entry of $\Gamma_lB_3(w_i)$, located, if $1\leq l\leq i-2$, at $\Sigma_l(M_3(v_i))$ and if $i\leq l\leq b-2$, at $\Sigma_{l+1}(M_3 (v_{i+1}))$, is transformed from $(-1)^{i+1}3$ into $1$, and in addition, the remaining non-zero entries of the column in question are zeroed.
\end{remark}

Remark \ref{rem417}, considered in terms of the blocks of $v_j$ and their columns, yields the equivalent
\begin{remark}\label{rem418}
According to Definition \ref{def416}, the only impact of operation $D_4$ on matrix $M_3$ is the following : for $3\leq j\leq b-1$ and  $1\leq t\leq b-1$ with $t\neq j-1$, all the non-zero entries of $\Sigma_t(M_3(v_j))$ are zeroed, excluding
\begin{itemize}
\item $(-1)^{j+1}3$, located at $\Gamma_tB_3(w_j)$, if $1\leq t\leq j-2$
\item $(-1)^{j}3$, located at $\Gamma_{t-1}B_3(w_{j-1})$, if $j\leq t\leq b-1$
\end{itemize}
which is transformed into 1.
\end{remark}

Now, we obtain matrix $M_4:=D_4M_3=D_4\dots D_1M$, given by Proposition \ref{pro41} for $n=4$. From (\ref{eq41a}) for $n=4$, (\ref{eq432}), (\ref{eq411}) - (\ref{eq414}), (\ref{eq336}) - (\ref{eq338}) and Remark \ref{rem418}, we obtain the following relations:
\begin{equation}\label{eq441}
A_4(u_i,v_j)=A(u_i,v_j),
\end{equation}  
where $2\leq i\leq b-1$ and $j=1$ or $3\leq j\leq b-1$ with $j\neq i+1$ or $j=i+1=b$, and $A(u_i,v_j)$ is given by (\ref{eq336}), (\ref{eq337}) and (\ref{eq338}), respectively,
\begin{equation}\label{eq442}
A_4(u_i,v_j)=A_1(u_i,v_j),\text{ where } 2\leq i\leq b-1 \text{ and } j=2 \text{ or } j=b\neq i+1,	
\end{equation}  
and $A_1(u_i,v_j)$ is given by (\ref{eq412}) and (\ref{eq413}), respectively, and
\begin{equation}\label{eq443} 
A_4(u_i,v_{i+1})=
\begin{bmatrix}
O_{1\times(i-1)}\quad-2\quad\quad O_{1\times(b-1-i)}
\end{bmatrix}, 
\text{where } 2\leq i\leq b-2.
\end{equation}

From (\ref{eq41b}) for $n=4$, (\ref{eq433}), (\ref{eq422}) - (\ref{eq425}), (\ref{eq339}) - (\ref{eq3312}), (\ref{eq416}) - (\ref{eq419}) and Remark \ref{rem418}, we obtain the following relations:
\begin{equation}\label{eq444}
B_4(w_i,v_j)=B(w_i,v_j),
\end{equation} 
where $2\leq i\leq b-1$ and $j=1$ or $3\leq j \leq b-1$ with $j\neq i$ and $j\neq i+1$, and $B(w_i,v_j)$ is given by (\ref{eq339}) and (\ref{eq3310}), respectively,
\begin{equation}\label{eq445}
B_4(w_i,v_j)=B_1(w_i,v_j),\text{ where } 2\leq i\leq b-1 \text{ and } j=2 \text{ or } j=b,
\end{equation} 
and $B_1(w_i,v_j)$ is given by (\ref{eq416}) - (\ref{eq417}) and (\ref{eq418}) - (\ref{eq419}), respectively,
\begin{equation}\label{eq446}
B_4(w_i,v_i)=
\begin{bmatrix}
I_{i-2}& & &\\
&(-1)^{i+1}2 &0 &\\
& & & O_{(b-2)-(i-1)}
\end{bmatrix},
\text{ where } 3\leq i\leq b-1, \text{and}
\end{equation}
\begin{equation}\label{eq447}
B_4(w_i,v_{i+1})=
\begin{bmatrix}
O_{i-2}& & &\\
&0 &(-1)^{i+1}2 &\\
& & & I_{(b-2)-(i-1)}
\end{bmatrix},
\text{ where } 2\leq i\leq b-2.
\end{equation}

From (\ref{eq41c}) for $n=4$, (\ref{eq434}) - (\ref{eq4312}), (\ref{eq3313}) - (\ref{eq3316}) and Remark \ref{rem418}, we obtain the following relations:
\begin{equation}\label{eq448}
C_4(w_{ij},v_h)=C(w_{ij},v_h),
\end{equation} 
where $2\leq i<j\leq b-1$ and $h=1$ or $3\leq h \leq b-1$ with $h\neq i$ and $h\neq j+1$, and $C(w_{ij},v_h)$ is given by  (\ref{eq3313}) and (\ref{eq3314}), respectively, 
\begin{equation}\label{eq449}
C_4(w_{ij},v_h)=C_3 (w_{ij},v_h), \text{ where } 2\leq i<j\leq b-1 \text{ and }  h=2 \text{ or } h=b,	
\end{equation}
 and $C_3(w_{ij},v_h)$ is given by (\ref{eq435}) - (\ref{eq437}) and (\ref{eq4310}) - (\ref{eq4312}), respectively,
\begin{equation}\label{eq4410}
C_4(w_{ib-1},v_i)=(-1)^b2E_{i-1}\in \mathbb{Z}^{(b-2)\times(b-1)}, \text{ where }  3\leq i<b-1,
\end{equation}
\begin{equation}\label{eq4411}
C_4(w_{2j},v_{j+1})=-2E_{j-1,j}\in \mathbb{Z}^{(b-2)\times(b-1)}, \text{ where } 2<j\leq b-2, \text{ and }
\end{equation}
\begin{equation}\label{eq4412}
C_4(w_{ij},v_h)=O_{(b-2)\times(b-1)},\text{where } 3\leq i<j\leq b-2 \text{ and } h=i \text{ or } h=j+1.
\end{equation}
\textbf{Result of $\boldsymbol{D_4}$.} Remark \ref{rem417} and the definitions of $M_4$ and $B_4$ yield the following
\begin{corollary}\label{cor419}
For $2\leq i\leq b-1$ and $1\leq l\leq b-2$ with $l\neq i-1$,  $\Gamma_lB_4(w_i)$ contains a unique non-zero entry equal to $1$, located, if $1\leq l\leq i-2$, at $\Sigma_l(M_4(v_i))$, if $i\leq l\leq b-2$, at $\Sigma_{l+1}(M_4(v_{i+1}))$, and this is the unique non-zero entry of the corresponding column.
\end{corollary}     
In other words, there are $(b-2)(b-3)$ entries of matrix $M_4$, each of which is equal to $1$ and constitutes the unique non-zero entry of the corresponding row and column.
\subsection{Operation $\boldsymbol{D_5}$}We consider matrix $M_4$ and focus on the block $B_4(w_i)$, where  $2\leq i\leq b-1$,
given by  (\ref{eq444}) - (\ref{eq447}). We already know that Corollary \ref{cor419} holds and observe that $\Gamma_{i-1}B_4(w_i)$ is non-zero for $2\leq i\leq b-1$. More specifically,
\begin{equation}\label{eq451}
\Gamma_{i-1}B_4(w_i,v_j)\neq O_{1\times (b-1)},
\end{equation} 
where $2\leq i\leq b-1$ and $j=2$ or $j=i\neq2$ or $j=i+1\neq b$ or $j=b$, while the remaining blocks of $\Gamma_{i-1}B_4(w_i)$ are null.
By operation $D_5$, we will null $\Gamma_{i-1}B_4(w_i)$, where $2\leq i\leq b-1$. We easily verify that
\begin{align}\label{eq452}
&\Gamma_{i-1}B_4(w_i)+\Gamma_iB_4(w_{i+1})+ (-1)^{i}\Gamma_iC_4(w_{2i+1}) +(-1)^{b+i}\Gamma_{i-1}C_4(w_{ib-1})\\
&\nonumber=O_{1\times\left[1+(b-1)^2\right]}, \text{where } 2\leq i \leq b-2, \text{ and}\\
\label{eq453}
&\Gamma_{b-2}B_4(w_{b-1})+(-1)^bA_4(u_{b-2})+(-1)^bA_4(u_{b-1})=O_{1\times\left[1+(b-1)^2\right]}.
\end{align}

\begin{definition}\label{def420}
We define operation $D_5$ as
\begin{equation*}
D_5:=\{D_5(w_i): 2\leq i\leq b-1\},
\end{equation*}
where
\begin{align*}
&D_5(w_i):\Gamma_{i-1}B_4(w_i)+ \Gamma_iB_4(w_{i+1})+(-1)^i\Gamma_iC_4 (w_{2i+1})+(-1)^{b+i}\Gamma_{i-1}C_4(w_{ib-1}),
\\
&\text{if } i\neq b-1,\text{ and } D_5(w_{b-1}):\Gamma_{b-2}B_4 (w_{b-1})+(-1)^bA_4(u_{b-2})+(-1)^bA_4(u_{b-1}).
\end{align*}
\end{definition}

From (4.5.1) - (4.5.3) and Definition \ref{def420}, we obtain the following
\begin{corollary}\label{cor421}
The unique change caused by applying operation $D_5$ to matrix $M_4$ is the nulling of     
$\Gamma_{i-1}B_4(w_i)$, where $2\leq i\leq b-1$. More specifically, $\Gamma_{i-1}B_4(w_i,v_j)$ is zeroed, where  $2\leq i\leq b-1$ and $j=2$ or $j=i\neq 2$ or $j=i+1\neq b$ or $j=b$.
\end{corollary}

Now, we obtain matrix $M_5:=D_5M_4=D_5\dots D_1M$, given by Proposition \ref{pro41} for $n=5$. From (\ref{eq41a}) for $n=5$ and Corollary \ref{cor421}, we conclude that
\begin{equation}\label{eq454}
A_5(u_i,v_j)=A_4(u_i,v_j), \text{ where } 2\leq i\leq b-1 \text{ and } 1\leq j\leq b,
\end{equation}  
and $A_4(u_i,v_j)$ is given by (\ref{eq441}) - (\ref{eq443}).   

From (\ref{eq41b}) for $n=5$, (\ref{eq444}) - (\ref{eq447}), (\ref{eq339}), (\ref{eq3310}), (\ref{eq416}) - (\ref{eq419}) and Corollary \ref{cor421}, we obtain the following relations:  
\begin{equation}\label{eq455}
B_5(w_i,v_1)=O_{(b-2)\times1} \text{ and } B_5(w_i,v_j)=O_{(b-2)\times(b-1)},
\end{equation}
where $2\leq i\leq b-1$ and $j=2$ or $3\leq j\leq b-1$ with $j\neq i$ and $j\neq i+1$ or $j=b$,
\begin{equation}\label{eq456}
B_5(w_i,v_i)=
\begin{bmatrix}
I_{i-2}&\\
& O_{(b-i)\times(b-i+1)}
\end{bmatrix}, 
\text{where } 3\leq i\leq b-1, \text{and}
\end{equation}
\begin{equation}\label{eq457}
B_5(w_i,v_{i+1})=
\begin{bmatrix}
O_{(i-1)\times i}&\\
& I_{(b-2)-(i-1)}
\end{bmatrix}, 
\text{ where } 2\leq i\leq b-2.
\end{equation}

From (\ref{eq41c}) for $n=5$ and Corollary \ref{cor421}, it follows that
\begin{equation}\label{eq458}
C_5(w_{ij},v_h)=C_4(w_{ij},v_h),\text{ where } 2\leq i<j\leq b-1 \text{ and } 1\leq h\leq b,
\end{equation} 
and $C_4(w_{ij},v_h)$ is given by (\ref{eq448}) - (\ref{eq4412}).\\
\textbf{Result of $\boldsymbol{D_5}$.} From (\ref{eq41b}) for $n=5$ and Corollary \ref{cor421}, we have that  $\Gamma_lB_5(w_i)=\Gamma_lB_4(w_i)$ and $\Gamma_{i-1}B_5(w_i)=O_{1\times\left[1+(b-1)^2\right]}$, where $2\leq i\leq b-1$ and $1\leq l\leq b-2$ with $l\neq i-1$. Considering Corollary \ref{cor419} as well, we obtain

\begin{proposition}\label{pro422}
For $2\leq i\leq b-1$ and $1\leq l\leq b-2$ with $l\neq i-1$,  $\Gamma_lB_5(w_i)$ contains a unique non-zero entry equal to $1$, located, if $1\leq l\leq i-2$, at $\Sigma_l(M_5(v_i))$, if $i\leq l\leq b-2$, at $\Sigma_{l+1}(M_5(v_{i+1}))$, and this is the unique non-zero entry of the corresponding column, while   $\Gamma_{i-1}B_5(w_i)=O_{1\times\left[1+(b-1)^2\right]}$.  
\end{proposition} 
Proposition \ref{pro422} implies that the blocks of $w_{i}$, where $2\leq i\leq b-1$, are completely simplified.

\subsection{Operation $\boldsymbol{D_6}$}We consider matrix $M_5$ and notice that for $2\leq i\leq b-3$, the block $A_5(u_i)$,    
given by (\ref{eq454}), is non-zero. More specifically,
\begin{align}\label{eq461}
&\text{for } 2\leq i\leq b-3, \text{ row }  A_5(u_i,v_j) \text{ is non-zero, where } j=1 \text{ or } j=2 \text{ or } j=\\
&\nonumber i+1 \text{ or } j=b, \text{ while the remaining blocks of }  A_5(u_i) \text{ are null.} 
\end{align}

By operation $D_6$, we will null $A_5(u_i)$, where $2\leq i\leq b-3$. We easily verify that
\begin{equation}\label{eq462}
A_5(u_i)-A_5(u_{i+2})+ \Gamma_{i+1}C_5(w_{2i+2})+ (-1)^b\Gamma_iC_5(w_{i+1b-1})=O_{1\times\left[ 1+(b-1)^2\right]},
\end{equation}
where $2\leq i\leq b-3$.

Now, we consider the block  $A_5(u_{b-1})$, also given by (\ref{eq454}). In the following, we will perform the appropriate row operations to null the units of matrix $A_5(u_{b-1},v_b)=
\begin{bmatrix}
a+1&-1&\ldots&-1&-2
\end{bmatrix}
\in\mathbb{Z}^{1\times(b-1)}$. This will facilitate us in the future operation  $D_9$.  We consider the rows $\Gamma_{i-1}C_5(w_{ib-1})$, where $2\leq i\leq b-2$ and $i\equiv b\text{mod}2$, given by (\ref{eq458}), and we sum them. We easily verify that
\begin{equation}\label{eq463}
A_5(u_{b-1},v_1)+(-1)^{b-1}\sum\limits_{\substack{i=2\\[2pt]i\equiv b\text{mod}2}}^{b-2}\Gamma_{i-1}C_5(w_{ib-1}, v_1)=(-1)^b[2]\in\mathbb{Z}^{1\times 1},
\end{equation}
\begin{align}\label{eq464}
&\text{for } 2\leq j\leq b-1,\text{ it holds }\\
\nonumber
A_5(u&_{b-1},v_j)+(-1)^{b-1}\sum\limits_{\substack{i=2\\[2pt]i\equiv b\text{mod}2}}^{b-2}
\Gamma_{i-1}C_5(w_{ib-1}, v_j)=
\begin{bmatrix}
O_{1\times(j-2)}\quad k\quad O_{1\times(b-j)}
\end{bmatrix},
\end{align} 
where $k=-2$, if $j\equiv b\text{mod}2$, and
$k=0$, if $j\not\equiv b\text{mod2}$, and
\begin{equation}\label{eq465}
A_5(u_{b-1},v_b)+(-1)^{b-1}\sum\limits_{\substack{i=2\\[2pt]i\equiv b\text{mod}2}}^{b-2}
\Gamma_{i-1}C_5(w_{ib-1},v_b)=
\begin{bmatrix}
\omega\quad O_{1\times(b-3)}\quad-2
\end{bmatrix},
\end{equation}
where $\omega=a+2$, if $b$ is even, and $\omega =a+1$, if $b$ is odd.

According to (\ref{eq465}), we can achieve the nulling of the units of matrix $A_5(u_{b-1}, v_b)$ by applying the operation   $A_5(u_{b-1})+(-1)^{b-1}\sum\limits_{\substack{i=2\\[2pt]i\equiv b\text{mod}2}}^{b-2}
\Gamma_{i-1}C_5(w_{ib-1})$.

\begin{definition}\label{def423}
We define operation $D_6$ as
\begin{equation*}
D_6:=\{D_6(u_i): 2\leq i\leq b-3 \text{ or } i=b-1\}
\end{equation*}
where
\begin{align*}
&D_6(u_i):A_5(u_i)-A_5(u_{i+2})+\Gamma_{i+1}C_5(w_{2i+2})+(-1)^b\Gamma_iC_5(w_{i+1b-1}),\\
&\text{where } 2\leq i\leq b-3, \text{ and }
D_6(u_{b-1}):A_5(u_{b-1})+(-1)^{b-1}\sum\limits_{\substack{i=2\\[2pt]i\equiv b\text{mod}2}}^{b-2}\Gamma_{i-1}C_5(w_{ib-1}).
\end{align*}
\end{definition}

From Definition \ref{def423}, (\ref{eq462}) and (\ref{eq461}), we obtain the following 
\begin{corollary}\label{cor424}
The only changes caused by applying operation $D_6$ to matrix $M_5$ are the nulling of  $A_5(u_i,v_j)$, where  $2\leq i\leq b-3$ and $j=1$ or $j=2$ or $j=i+1$ or $j=b$, and the transformation of $A_5(u_{b-1})$  as described by $\mathrm{(\ref{eq463})}$ - $\mathrm{(\ref{eq465})}$.
\end{corollary}

Now, we obtain matrix $M_6:=D_6M_5=D_6\dots D_1M$, given by Proposition \ref{pro41} for $n=6$. From (\ref{eq41a}) for $n=6$, (\ref{eq454}), Corollary \ref{cor424} and (\ref{eq441}) - (\ref{eq443}), we obtain the following relations:
\begin{equation}\label{eq466}
A_6(u_i,v_1)=O_{1\times1}  \text{ and }  A_6(u_i,v_j)=O_{1\times(b-1)}, \text{where } 2\leq i\leq b-3 \text{ and }  2\leq j\leq b,
\end{equation}
\begin{equation}\label{eq467}
A_6(u_{b-2},v_j)=A(u_{b-2},v_j), \text{ where } j=1 \text{ or } 3\leq j\leq b-2,
\end{equation}
and $A(u_{b-2},v_j)$ is given by (\ref{eq336}) and (\ref{eq337}), respectively,
\begin{equation}\label{eq468}
A_6(u_{b-2},v_j)=A_1(u_{b-2},v_j), \text{ where } j=2 \text{ or } j=b,	
\end{equation}
and $A_1(u_{b-2},v_j)$ is given by (\ref{eq412}) and (\ref{eq413}), respectively, and
\begin{equation}\label{eq469}
A_6(u_{b-2},v_{b-1})=A_4(u_{b-2},v_{b-1}), \text{where } A_4(u_{b-2},v_{b-1}) \text{ is given by } (\ref{eq443}).	
\end{equation}

From (\ref{eq41a}) for $n=6$, Definition \ref{def423}  and  (\ref{eq463}) - (\ref{eq465}), it follows that
\begin{equation}\label{eq4610}
A_6(u_{b-1},v_1)=(-1)^b[2]\in\mathbb{Z}^{1\times1},
\end{equation}
\begin{equation}\label{eq4611}
A_6(u_{b-1},v_j)=
\begin{bmatrix}
O_{1\times(j-2)}\quad k\quad O_{1\times(b-j)}
\end{bmatrix},
\end{equation} 
\text{where } $2\leq j\leq b-1$, $k=-2$, if $j\equiv b\text{mod}2$, and $k=0$, if $j \not\equiv b\text{mod2}$, and
\begin{equation}\label{eq4612}
A_6(u_{b-1},v_b)=
\begin{bmatrix}
\omega\quad O_{1\times(b-3)}\quad-2
\end{bmatrix},
\end{equation}
where $\omega=a+2$, if  $b$ is even, and   $\omega=a+1$, if $b$ is odd.

From (\ref{eq41b}) and (\ref{eq41c}) for $n=6$, Corollary \ref{cor424} and (\ref{eq458}), we conclude that
\begin{equation}\label{eq4613}
B_6(w_i,v_j)=B_5(w_i,v_j), \text{ where } 2\leq i\leq b-1 \text{ and } 1\leq j\leq b, \end{equation}
and $B_5(w_i,v_j)$ is given by (\ref{eq455}) - (\ref{eq457}), and 
\begin{equation}\label{eq4614}
C_6(w_{ij},v_h)=C_4(w_{ij},v_h), \text{ where } 2\leq i<j\leq b-1 \text{ and } 1\leq h\leq b, 
\end{equation}
and $C_4(w_{ij},v_h)$ is given by (\ref{eq448}) - (\ref{eq4412}).\\
\textbf{Results of $\boldsymbol{D_{6}}$.} Relation (\ref{eq466}) yields
\begin{equation}\label{eq4615}
A_6(u_i)=O_{1\times\left[1+(b-1)^2\right]}, \text{where } 2\leq i\leq b-3.
\end{equation}
Furthermore, (\ref{eq4612}) will facilitate us in the application of operation $D_9$.

\subsection{Operation $\boldsymbol{D_7}$}
We consider matrix $M_6$ and  focus on $\Gamma_{j-1}C_6(w_{2j})$, where $3\leq j\leq b-1$, given by (\ref{eq4614}), (\ref{eq448}), (\ref{eq3313}), (\ref{eq3314}), (\ref{eq449}), (\ref{eq435}), (\ref{eq436}), (\ref{eq4310}), (\ref{eq4311}) and (\ref{eq4411}). 

\begin{remark}\label{rem425}
In matrix $M_6$, we observe that
for $3\leq j\leq b-1$,  $\Gamma_{j-1}C_6(w_{2j})$ contains exactly three non-zero entries, 
$(-1)^{j+1}$, $(-1)^{j+1}$ and $-2$, located at $\Sigma_{j-1}\left(M_6(v_2)\right)$, $\Sigma_j\left(M_6(v_2)\right)$ and     
$\Sigma_j\left(M_6(v_{j+1})\right)$, \text{respectively}.
\end{remark}

For $3\leq j\leq b-1$,  we will null the entries $(-1)^{j+1}$ and $-2$ of $\Gamma_{j-1}C_6(w_{2j})$, 
located at $\Sigma_j\left(M_6(v_2)\right)$ and $\Sigma_j\left(M_6(v_{j+1})\right)$, respectively.

Now, we focus on the block  $M_6(v_2)$, given by (\ref{eq466}), (\ref{eq468}), (\ref{eq412}), (\ref{eq4611}), (\ref{eq4613}), (\ref{eq455}), (\ref{eq4614}), (\ref{eq449}) and (\ref{eq435}) - (\ref{eq437}). 

\begin{remark}\label{rem426}
In the block $M_6(v_2)$, we notice the following:
\begin{enumerate}
\item $\Sigma_2\left(M_6(v_2)\right)$ contains a unique non-zero entry equal to 1, located at $\Gamma_2C_6(w_{23})$,
\item for $3\leq j\leq b-2$, $\Sigma_j\left(M_6(v_2)\right)$ contains exactly two non-zero entries, $(-1)^{j+1}$ and  $(-1)^j$, located at $\Gamma_{j-1}C_6(w_{2j})$ and $\Gamma_jC_6(w_{2j+1})$, respectively, and 
\item $\Sigma_{b-1}\left(M_6(v_2)\right)$ contains exactly two non-zero entries, each equal to $(-1)^b$, located at rows $A_6(u_{b-2})$ and $\Gamma_{b-2}C_6(w_{2b-1})$, respectively.
\end{enumerate} 
\end{remark}

\begin{proposition}\label{pro427}
For\; $3\leq j\leq b-1$, the column sum  $\sum\limits_{\tau=2}^{j-1}(-1)^{\tau+j-1}
\Sigma_\tau\left(M_6(v_2)\right)$ contains a unique non-zero entry equal to $(-1)^{j+1}$, located at the $(j-1)-th$ row of the block of $w_{2j}$.
\end{proposition}
\begin{proof}We proceed by induction on $j$ $(3\leq j\leq b-1)$. Let $j=3$. Then, we have
\begin{equation*}
\sum\limits_{\tau=2}^{j-1}(-1)^{\tau+j-1} \Sigma_\tau\left(M_6(v_2)\right)=\Sigma_2\left(M_6(v_2)\right).
\end{equation*}
Considering also Remark \ref{rem426}(1), we conclude that Proposition \ref{pro427} holds for $j=3$. Let $3\leq j\leq b-2$. We assume that (induction hypothesis) Proposition  \ref{pro427} holds for $j$. We will show that Proposition \ref{pro427} holds for $j+1$. It is easy to verify that
\begin{equation*}
\sum\limits_{\tau=2}^{(j+1)-1}(-1)^{\tau+(j+1)-1}\Sigma_\tau\left(M_6(v_2)\right)=\Sigma_j\left(M_6(v_2)\right)-\sum\limits_{\tau=2}^{j-1}(-1)^{\tau+j-1}\Sigma_\tau\left(M_6(v_2)\right),
\end{equation*}
which combined with Remark \ref{rem426}(2) and the induction hypothesis for $j$, implies that Proposition \ref{pro427} holds for $j+1$. Now, the proof is compete.
\end{proof}

Later, we will need the following general
\begin{remark}\label{rem428}
Let $M\in\mathbb{Z}^{m\times n}$ $(m,n\in\mathbb{N})$ be a matrix, and let $\Sigma_*$ be a column of $M$ that contains a unique non-zero entry equal to $(-1)^i$ $(i\in\mathbb{N})$, located at row $\Gamma_*$ of $M$. By applying the appropriate column operations, we null the remaining non-zero entries of $\Gamma_*$. Indeed, we consider those columns of $M$, each containing a non-zero entry at $\Gamma_*$, except $\Sigma_*$. By adding or subtracting an appropriate multiple of $\Sigma_*$ from each of the above columns, we null all the non-zero entries of  $\Gamma_*$, except  $(-1)^i$, located at $\Sigma_*$. Obviously, this is the unique change on $M$ caused by the above procedure.
\end{remark}

\begin{definition}\label{def429}
We define operation $D_7$ as
\begin{equation*}
D_7:=\{D_7(v_2,j),\ D_7(v_{j+1},j):\ 3\leq j\leq b-1\},
\end{equation*}
where for  $3\leq j\leq b-1$,
\begin{align*}
&D_7(v_2,j): \Sigma_j\left(M_6(v_2)\right)-\sum\limits_{\tau =2}^{j-1}(-1)^{\tau+j-1}
\Sigma_\tau\left(M_6(v_2)\right),\text{ and}\\
&D_7(v_{j+1},j): \Sigma_j\left(M_6(v_{j+1})\right)+(-1)^{j-1}2\sum\limits_{\tau=2}^{j-1}(-1)^{\tau+j-1}\Sigma_\tau\left(M_6(v_2)\right).
\end{align*}
\end{definition}

Definition \ref{def429}, Remark \ref{rem425}, Proposition \ref{pro427} and Remark \ref{rem428} lead to
\begin{corollary}\label{cor430}
The unique change caused by applying operation $D_7$ to matrix $M_6$ is the following: for $3\leq j\leq b-1$, the entries $(-1)^{j+1}$ and $-2$ of $\Gamma_{j-1}C_6(w_{2j})$, located at  $\Sigma_j\left(M_6(v_2)\right)$ and $\Sigma_j\left(M_6(v_{j+1})\right)$, respectively, are zeroed.
\end{corollary}

Now, we obtain matrix $M_7:=D_7M_6=D_7\dots D_1M$, given by Proposition \ref{pro41}  for $n=7$. From (\ref{eq41a}) and (\ref{eq41b}) for $n=7$, Corollary \ref{cor430} and (\ref{eq4613}), we have
\begin{equation}\label{eq471}
A_7(u_i,v_j)=A_6(u_i,v_j), \text{ where } 2\leq i\leq b-1 \text{ and } 1\leq j\leq b,
\end{equation}
and $A_6(u_i,v_j)$ is given by (\ref{eq466}) - (\ref{eq4612}), and
\begin{equation}\label{eq472}
B_7(w_i,v_j)=B_5(w_i,v_j), \text{ where } 2\leq i\leq b-1 \text{ and } 1\leq j\leq b,
\end{equation}
and $B_5(w_i,v_j)$ is given by (\ref{eq455}) - (\ref{eq457}).

From (\ref{eq41c}) for $n=7$, (\ref{eq4614}), (\ref{eq448}) - (\ref{eq4412}), (\ref{eq3313}), (\ref{eq3314}), (\ref{eq435}) - (\ref{eq437}), (\ref{eq4310}) - (\ref{eq4312}) and Corollary \ref{cor430}, we obtain the following relations:
\begin{equation}\label{eq473}
C_7(w_{ij},v_1)=O_{(b-2)\times1}, \text{ where } 2\leq i<j\leq b-1,
\end{equation}
\begin{equation}\label{eq474}
C_7(w_{ij},v_h)=O_{(b-2)\times(b-1)},
\end{equation}
where $3\leq i<j\leq b-1$ and $h=2$  or  $2\leq i<j\leq b-1$ and $3\leq h\leq b-1$ with $h\neq i$ or $3\leq i<j\leq b-2$ and $h=i$ or $2\leq i<j\leq b-2$ and $h=b$,
\begin{equation}\label{eq475}
C_7(w_{2j},v_2)=(-1)^{j+1}E_{j-1}\in\mathbb{Z}^{(b-2)\times(b-1)}, \text{ where } 3\leq j\leq b-2,
\end{equation}
\begin{equation}\label{eq476}
C_7(w_{2b-1},v_2)=(-1)^b(2E_1+E_{b-2})\in\mathbb{Z}^{(b-2)\times(b-1)},
\end{equation}
\begin{equation}\label{eq477}
C_7(w_{ib-1},v_i)=(-1)^b2E_{i-1}\in\mathbb{Z}^{(b-2)\times(b-1)}, \text{ where } 3\leq i\leq b-2, \text{ and}
\end{equation}
\begin{equation}\label{eq478}
C_7(w_{ib-1},v_b) =(-1)^{i+1}(E_{i-1}+E_{i-1,i})\in\mathbb{Z}^{(b-2)\times(b-1)}, \text{ where } 2\leq i\leq b-2.
\end{equation}
\textbf{Results of $\boldsymbol{D_{7}}$.} From Remarks \ref{rem425} - \ref{rem426}, Proposition \ref{pro427}, Remark \ref{rem428}, Definition \ref{def429} and Corollary \ref{cor430}, we conclude the following
\begin{corollary}\label{cor431}
Matrix $M_7$ has the following properties:
\begin{enumerate}
\item for $3\leq j\leq b-1$, $\Gamma_{j-1}C_7(w_{2j})$ contains a unique non-zero entry equal to $(-1)^{j+1}$, located at $\Sigma_{j-1}\left(M_7(v_2)\right)$, and this is the unique non-zero entry of $\Sigma_{j-1}\left(M_7(v_2)\right)$, and
\item $\Sigma_{b-1}\left(M_7(v_2)\right)$ contains a unique non-zero entry equal to $(-1)^b$, located at row $A_{7}(u_{b-2})$.
\end{enumerate}
\end{corollary}

\subsection{Operation $\boldsymbol{D_8}$}
We consider matrix $M_7$ that satisfies Corollary \ref{cor431}(2). We identify the non-zero entries of $A_7(u_{b-2})$, which is given by (\ref{eq471}), (\ref{eq467}) - (\ref{eq469}), (\ref{eq336}), (\ref{eq337}), (\ref{eq412}), (\ref{eq413}) and (\ref{eq443}), and we apply the procedure described by Remark \ref{rem428}. The above leads to the following

\begin{definition}\label{def432}
We define operation $D_8$ as
\begin{equation*}
D_8:=\{D_8(v_1),\ D_8(v_{b-1}),\ D_8(v_b,\tau): 1\leq\tau\leq b-3\},
\end{equation*}
where
\begin{align*}
&D_8(v_1):M_7(v_1)+2\Sigma_{b-1}\left(M_7(v_2)\right),\\
&D_8(v_{b-1}):\Sigma_{b-2}\left(M_7(v_{b-1})\right)+(-1)^b2\Sigma_{b-1}\left(M_7(v_2)\right),\\
&D_8(v_b,1):\Sigma_1\left(M_7(v_b)\right)+(-1)^b(a+1)\Sigma_{b-1}\left(M_7(v_2)\right),\text{ and}\\
&D_8(v_b,\tau):\Sigma_\tau\left(M_7(v_b)\right)+(-1)^{b-1}\Sigma_{b-1}\left(M_7(v_2)\right), \text{ where } 2\leq\tau\leq b-3.
\end{align*}
\end{definition}

From Definition \ref{def432} and the comment before it, we conclude 
\begin{corollary}\label{cor433}
The unique change caused by applying operation $D_8$ to matrix $M_7$ is the nulling of all non-zero entries of row $A_7(u_{b-2})$, excluding $(-1)^b$, located at $\Sigma_{b-1}\left(M_7(v_2)\right)$. 
\end{corollary}

Now, we obtain matrix $M_8:=D_8M_7=D_8\dots D_1M$, given by Proposition \ref{pro41} for $n=8$. From (\ref{eq41a}) for $n=8$, Corollary \ref{cor433}, (\ref{eq471}), (\ref{eq466}), (\ref{eq468}) and (\ref{eq412}), it follows that
\begin{equation}\label{eq481}
A_8(u_i,v_1)=O_{1\times1} \text{ and } A_8(u_i,v_j)=O_{1\times(b-1)}, 
\end{equation} 
where  $2\leq i\leq b-3$ and $2\leq j\leq b$,
\begin{align}\label{eq482}
&A_8(u_{b-2},v_1)=O_{1\times1},\  A_8(u_{b-2},v_2)=(-1)^b
\begin{bmatrix}
O_{1\times(b-2)}\quad1
\end{bmatrix} \text{ and}\\\nonumber
&A_8(u_{b-2},v_j)=O_{1\times(b-1)}, \text{ where } 3\leq j\leq b, \text{ and}
\end{align} 
\begin{equation}\label{eq483}
A_8(u_{b-1},v_j)=A_6(u_{b-1},v_j), \text{ where } 1\leq j\leq b,
\end{equation}
and $A_6(u_{b-1},v_j)$ is given by (\ref{eq4610}) - (\ref{eq4612}).

From (\ref{eq41b}) and (\ref{eq41c}) for $n=8$, Corollary \ref{cor433} and (\ref{eq472}), it follows that 
\begin{equation}\label{eq484}
B_8(w_i,v_j)=B_5(w_i,v_j),\text{ where } 2\leq i\leq b-1 \text{ and } 1\leq j\leq b,
\end{equation} 
and $B_5(w_i,v_j)$ is given by (\ref{eq455}) - (\ref{eq457}), and
\begin{equation}\label{eq485}
C_8(w_{ij},v_h)=C_7(w_{ij},v_h), \text{ where }
2\leq i<j\leq b-1 \text{ and } 1\leq h\leq b,
\end{equation}  
and $C_7(w_{ij},v_h)$ is given by (\ref{eq473}) - (\ref{eq478}).\\
\textbf{Result of $\boldsymbol{D_8}$.} Corollaries \ref{cor431}(2) and \ref{cor433} yield the following 
\begin{corollary}\label{cor434} 
Row $A_8(u_{b-2})$ contains a unique non-zero entry equal to $(-1)^b$, located at $\Sigma_{b-1}(M_8(v_2))$, and this is the unique non-zero entry of $\Sigma_{b-1}(M_8(v_2))$.
\end{corollary} 

\subsection{Operation $\boldsymbol{D_9}$}
We consider matrix $M_8$ and focus on  $\Gamma_{i-1}C_8(w_{ib-1})$, where $2\leq i \leq b-2$, given by (\ref{eq485}), (\ref{eq473}), (\ref{eq474}) and  (\ref{eq476}) - (\ref{eq478}). 
\begin{remark}\label{rem435}
In matrix $M_8$, we notice that for $2\leq i\leq b-2$, $\Gamma_{i-1}C_8(w_{ib-1})$ contains exactly three non-zero entries, $(-1)^b2$,  $(-1)^{i+1}$ and $(-1)^{i+1}$, located at $\Sigma_{i-1}(M_8(v_i))$,  $\Sigma_{i-1}(M_8(v_b))$ and   $\Sigma_i(M_8(v_b))$, respectively.
\end{remark}

For $2\leq i\leq b-2$, we will null the entries $(-1)^b2$ and $(-1)^{i+1}$ of $\Gamma_{i-1}C_8(w_{ib-1})$, located at $\Sigma_{i-1}(M_8(v_i))$ and $\Sigma_{i-1}(M_8(v_b))$, respectively.

Now, we focus on the block $M_8(v_b)$, given by (\ref{eq481}) - (\ref{eq483}), (\ref{eq4612}), (\ref{eq484}), (\ref{eq455}), (\ref{eq485}), (\ref{eq474})  and (\ref{eq478}). 
\begin{remark}\label{rem436}
In the block $M_8(v_b)$, we observe the following:
\begin{enumerate}
\item $\Sigma_1(M_8(v_b))$ contains exactly two non-zero entries, $\omega$ [see (\ref{eq4612})] and $-1$, located at rows $A_8(u_{b-1})$ and   $\Gamma_1C_8(w_{2b-1})$, respectively,
\item for $2\leq i\leq b-3$, $\Sigma_i(M_8(v_b))$ contains exactly two non-zero entries, $(-1)^{i+1}$ and $(-1)^i$, located at $\Gamma_{i-1}C_8(w_{ib-1})$ and $\Gamma_iC_8(w_{i+1b-1})$, respectively, and
\item $\Sigma_{b-2}(M_8(v_b))$ contains a unique non-zero entry equal to $(-1)^{b-1}$, located at $\Gamma_{b-3}C_8(w_{b-2b-1})$.
\end{enumerate} 
\end{remark}

\begin{proposition}\label{pro437}
For $2\leq i\leq b-2$, the column sum $\sum\limits_{\tau=i}^{b-2}(-1)^{\tau+i}\Sigma_\tau\left(M_8(v_b)\right)$ contains a unique non-zero entry equal to $(-1)^{i+1}$, located at $\Gamma_{i-1}C_8(w_{ib-1})$.
\end{proposition}
\begin{proof}By induction on $r=(b-1)-i$ $(1\leq r\leq b-3)$, similarly to the proof of Proposition \ref{pro427}.
\end{proof}

\begin{definition}\label{def438}
We define operation $D_9$ as
\begin{equation*}
D_9:=\{D_9(v_i,i-1),\ D_9(v_b,i-1): 2\leq i\leq b-2\},
\end{equation*}
where for $2\leq i\leq b-2$,
\begin{align*}
&D_9(v_i,i-1): \Sigma_{i-1}(M_8(v_i))+(-1)^{b+i}2 \sum\limits_{\tau=i}^{b-2}(-1)^{\tau+i}  \Sigma_\tau(M_8(v_b)),\text{ and }\\
&D_9(v_b,i-1):\Sigma_{i-1}(M_8(v_b))- \sum\limits_{\tau=i}^{b-2}(-1)^{\tau+i}\Sigma_\tau(M_8(v_b)).
\end{align*}
\end{definition}

Definition \ref{def438}, Remark \ref{rem435}, Proposition \ref{pro437} and Remark \ref{rem428}, lead to
\begin{corollary}\label{cor439}
The unique change caused by applying operation $D_9$ to matrix $M_8$ is the following: for $2\leq i\leq b-2$, the entries $(-1)^b2$ and $(-1)^{i+1}$ of $\Gamma_{i-1}C_8(w_{ib-1})$, located at $\Sigma_{i-1}(M_8(v_i))$ and $\Sigma_{i-1}(M_8(v_b))$, respectively, are zeroed.
\end{corollary} 

Now, we obtain matrix $M_9:=D_9M_8=D_9\dots D_1M$, given by Proposition \ref{pro41} for $n=9$. From (\ref{eq41a}) and (\ref{eq41b}) for $n=9$, Corollary \ref{cor439} and (\ref{eq484}), we have
\begin{equation}\label{eq491}
A_9(u_i,v_j)=A_8(u_i,v_j), \text{ where } 2\leq i\leq b-1 \text{ and } 1\leq j\leq b,
\end{equation} 
and $A_8(u_i,v_j)$ is given by (\ref{eq481}) - (\ref{eq483}), and
\begin{equation}\label{eq492}
 B_9(w_i,v_j)=B_5(w_i,v_j), \text{ where }  2\leq i\leq b-1 \text{ and } 1\leq j\leq b,
\end{equation}  
and $B_5(w_i,v_j)$ is given by (\ref{eq455}) - (\ref{eq457}).

From (\ref{eq41c}) for $n=9$, (\ref{eq485}), (\ref{eq473}) - (\ref{eq478}) and Corollary \ref{cor439}, we obtain the following relations:
\begin{equation}\label{eq493}
C_9(w_{ij},v_1)=O_{(b-2)\times1} \text{ and } C_9(w_{ij},v_h)=O_{(b-2)\times(b-1)}, 
\end{equation}
where $2\leq i<j\leq b-1$ and $h=2\neq i$ or $3\leq h\leq b-1$ or $h=b\neq j+1$,
\begin{equation}\label{eq494}
C_9(w_{2j},v_2)=(-1)^{j+1}E_{j-1}\in\mathbb{Z}^{(b-2)\times(b-1)}, \text{ where } 3\leq j\leq b-1, \text{ and}
\end{equation}
\begin{equation}\label{eq495}  
C_9(w_{ib-1},v_b)=(-1)^{i+1}E_{i-1,i}\in\mathbb{Z}^{(b-2)\times(b-1)},\text{ where } 2\leq i\leq b-2. 
\end{equation}
\textbf{Results of $\boldsymbol{D_9}$.} 
From Remarks \ref{rem435} - \ref{rem436}, Proposition \ref{pro437}, Definition \ref{def438} and Corollary \ref{cor439}, we deduce 
\begin{corollary}\label{cor440}
Matrix $M_9$ has the following properties:
\begin{enumerate}
\item for  $2\leq i\leq b-2$,  $\Gamma_{i-1}C_9(w_{ib-1})$ contains a unique non-zero entry equal to $(-1)^{i+1}$, located at $\Sigma_i(M_9(v_b))$, and this is the unique non-zero entry of $\Sigma_i(M_9(v_b))$, and
\item $\Sigma_1(M_9(v_b))$ contains a unique non-zero entry equal to $a+2$, if $b$ is even, and $a+1$, if $b$ is odd, located at row $A_9(u_{b-1})$.
\end{enumerate}
\end{corollary} 

\subsection{Operation $\boldsymbol{D_{10}}$}
We consider matrix $M_9$ and focus on column $M_9(v_1)$, given by (\ref{eq491}), (\ref{eq481}) - (\ref{eq483}), (\ref{eq4610}), (\ref{eq492}), (\ref{eq455}) and (\ref{eq493}). 
\begin{remark}\label{rem441}
We observe that	column $M_9(v_1)$ contains a unique non-zero entry equal to  $(-1)^b2$, located at row $A_9(u_{b-1})$. 
\end{remark}

Now, we focus on row $A_9(u_{b-1})$, given by  (\ref{eq491}), (\ref{eq483}) and (\ref{eq4610}) - (\ref{eq4612}). We consider those columns of $M_9$, each containing an entry equal to $\pm2$ at row $A_9(u_{b-1})$, excluding column  $M_9(v_1)$. We add or subtract an appropriate multiple of column $M_9(v_1)$ from each of the above columns so that all the entries of row  $A_9(u_{b-1})$ that are equal to $\pm2$, excluding $(-1)^b2$ in column $M_9(v_1)$, are zeroed. Because of Remark \ref{rem441}, this is the only change caused by the above procedure on matrix $M_9$.
\begin{definition}\label{def442}
We define operation $D_{10}$ as
\begin{equation*}
D_{10}:=\{D_{10}(v_j): 2\leq j\leq b \text{ and } j\equiv b\text{mod }2\},
\end{equation*}
where
\begin{equation*}
D_{10}(v_j):\Sigma_{j-1}\left(M_9(v_j) \right)+(-1)^bM_9(v_1), \text{ where } 2\leq j\leq b \text{ and } j\equiv b\text{mod }2.
\end{equation*}
\end{definition}

Definition \ref{def442}, Remark \ref{rem441} and the comment after it, lead to
\begin{corollary}\label{cor443}
The unique change caused by applying operation $D_{10}$ to matrix $M_9$ is the nulling of $-2$ at the matrices  $A_9(u_{b-1},v_j)$, where $2\leq j\leq b$ and $j\equiv b\text{mod }2$.
\end{corollary}

Now, we obtain matrix $M_{10}:=D_{10}M_9=D_{10}\dots D_1M$, given by Proposition \ref{pro41} for $n=10$. From (\ref{eq41a}) for $n=10$, (\ref{eq491}), (\ref{eq481}) - (\ref{eq483}), (\ref{eq4610}) - (\ref{eq4612}) and  Corollary \ref{cor443}, we obtain the following relations:
\begin{equation}\label{eq4101}
A_{10}(u_i,v_1)=O_{1\times1}, \text{ where } 2\leq i\leq b-2, \text{ and } A_{10}(u_{b-1}, v_1)=(-1)^b[2],
\end{equation}
\begin{equation}\label{eq4102} 
A_{10}(u_i,v_j)=O_{1\times(b-1)}, 
\end{equation}
where $2\leq i\leq b-3$ and $2\leq j\leq b$ or $i=b-2$ and $3\leq j\leq b$ or $i=b-1$ and $2\leq j\leq b-1$,
\begin{equation}\label{eq4103}
A_{10}(u_{b-2},v_2)=(-1)^b
\begin{bmatrix}
O_{1\times(b-2)}\quad1
\end{bmatrix}, \text{ and} 
\end{equation}
\begin{equation}\label{eq4104}
A_{10}(u_{b-1},v_b)=
\begin{bmatrix}
\omega\quad O_{1\times(b-2)}
\end{bmatrix},  
\end{equation}
where $\omega=a+2$, if $b$ is even, and $\omega =a+1$, if $b$ is odd.

From (\ref{eq41b}) and (\ref{eq41c}) for $n=10$, Corollary \ref{cor443} and (\ref{eq492}), we conclude that
\begin{equation}\label{eq4105}
B_{10}(w_i,v_j)=B_5(w_i,v_j), \text{ where } 2\leq i\leq b-1 \text{ and } 1\leq j\leq b,
\end{equation}
and $B_5(w_i,v_j)$ is given by (\ref{eq455}) - (\ref{eq457}), and
\begin{equation}\label{eq4106}
C_{10}(w_{ij},v_h)=C_9(w_{ij},v_h), \text{ where } 2\leq i<j \leq b-1 \text{ and }
1\leq h\leq b,
\end{equation}
and $C_9(w_{ij},v_h)$ is given by (\ref{eq493}) - (\ref{eq495}).

From (\ref{eq4101}) - (\ref{eq4106}), (\ref{eq455}) - (\ref{eq457}) and (\ref{eq493}) - (\ref{eq495}), we obtain the following 
\begin{corollary}\label{cor444}
Matrix $M_{10}$ has the following properties:
\begin{enumerate}
\item $A_{10}(u_i)=O_{1\times[1+(b-1)^2]}$, where $2\leq i\leq b-3$,
\item row $A_{10}(u_{b-2})$ contains a unique non-zero entry equal to $(-1)^b$, located at $\Sigma_{b-1}(M_{10}(v_2))$, and this is the unique non-zero entry of $\Sigma_{b-1}(M_{10}(v_2))$,
\item row $A_{10}(u_{b-1})$ contains exactly two non-zero entries, $(-1)^b2$ and $\omega$  $\mathrm{[(\ref{eq4104})]}$, located at column $M_{10}(v_1)$ and $\Sigma_1(M_{10}(v_b))$, respectively; furthermore, each of the above entries is the unique non-zero entry of the corresponding column,
\item for $2\leq i\leq b-1$ and $1\leq l\leq b-2$ with $l\neq i-1$, $\Gamma_lB_{10}(w_i)$ contains a unique non-zero entry equal to $1$, located, if  $1\leq l\leq i-2$, at $\Sigma_l(M_{10}(v_i))$, if $i\leq l\leq b-2$, at $\Sigma_{l+1}(M_{10}(v_{i+1}))$, and this is the unique non-zero entry of the corresponding column,
\item $\Gamma_{i-1}B_{10}(w_i)=
O_{1\times[1+(b-1)^2]}$, where $2\leq i\leq b-1$,
\item for $3\leq j\leq b-1$,  $\Gamma_{j-1}C_{10}(w_{2j})$ contains a unique non-zero entry equal to $(-1)^{j+1}$, located at $\Sigma_{j-1}(M_{10}(v_2))$, and this is the unique non-zero entry of $\Sigma_{j-1}(M_{10}(v_2))$, and
\item for $2\leq i\leq b-2$,  $\Gamma_{i-1}C_{10}(w_{ib-1})$ contains a unique non-zero entry equal to  $(-1)^{i+1}$, located at  $\Sigma_i(M_{10}(v_b))$, and this is the unique non-zero entry of  $\Sigma_i(M_{10}(v_b))$.
\end{enumerate} 
\end{corollary}

From Corollary \ref{cor444} (2), (4), (6) and  (7), we obtain the following
\begin{corollary}\label{cor445}
In matrix $M_{10}$, there exist $b(b-3)+1$ entries, each of which is equal to $\pm 1$ and constitutes the unique non-zero entry of the corresponding row and column.
\end{corollary}

\subsection{Operation $\boldsymbol{D_{11}}$}
We consider matrix $M_{10}$ and focus on Corollary \ref{cor444}. We apply the appropriate interchanges of rows and columns, and the necessary changes of sign to rows and columns so that $M_{10}$  is transformed into the equivalent matrix
\begin{equation*}
M_{11}=
\begin{bmatrix}
I_r& &&\\
& 2&\omega&\\
&&&O
\end{bmatrix}\in\mathbb{Z}^{s\times t},
\end{equation*}
where $r=b(b-3)+1$, $\omega=a+2$, if $b$ is even, $\omega=a+1$, if $b$ is odd,  $s=(b-2)\left[1+\binom{b-1}{2}\right]$ and $t=1+(b-1)^2$.

In the following, we need the well-known
\begin{proposition}\label{pro446}
Let $A\in R^{m\times n}$ $(m,n\in\mathbb{N})$
be a matrix, where $R$ is a principal ideal domain, and $\mathrm{diag}(d_1,\dots,d_u)$ is a Smith normal form of $A$. Let $\delta_i$ be a g.c.d of the determinants of the  $i\times i$ submatrices of $A$. Then, there are invertible elements  $c_1,\dots,c_u\in R$ such that 
$d_1=c_1\delta_1$,  $d_2\delta_1=c_2\delta_2$, \dots, $d_u\delta_{u-1}=c_u\delta_u$.
\end{proposition}

Using Proposition \ref{pro446}, we compute that the invariant factors of the matrix $M$, which coincide with the invariant factors of the equivalent matrix $M_{11}$, are
\begin{equation*}
d_i =\pm1, \text{ for } 1\leq i\leq r,\;
d_{r+1}=\begin{cases}
\pm2,\text { if } a\equiv b\bmod2\\ 
\pm1,\text{ if } a\not\equiv b\bmod2
\end{cases}
\text{and } d_i = 0,\text{ for } r+2\leq i\leq t.
\end{equation*}
From the above, (\ref{eq333}) and (\ref{eq334}), it follows that
\begin{equation*} \mathrm{Ext}_A^2(K_{\lambda}F,K_{\mu}F)\cong\mathbb{Z}_{d_1}\oplus\ldots\oplus\mathbb{Z}_{d_r}\oplus\mathbb{Z}_{d_{r+1}}\cong\mathbb{Z}_{d_{r+1}}=\begin{cases}
\mathbb{Z}_2, &\text{if } a\equiv b\bmod2\\ 0, &\text{if } a\not\equiv b\bmod2,
\end{cases}
\end{equation*}
hence we have proven Theorem \ref{thm33}(1) for $b\geq7$.

\section{The invariant factors of matrix $M$ for $3\leq b\leq 6$}
In this section, we determine $\mathrm{Ext}_A^2(K_{\lambda}F,K_{\mu}F)$, where $\lambda=(a,1^b)$ and $\mu=(a+1,b-1)$, for $b=3,\dots,6$ and $a+1>b-1$. We recall that the matrix $M=M(a,b)$ of $\mathrm{Ext}_A^2(K_{\lambda}F,K_{\mu}F)$ is given by (\ref{eq335}) - (\ref{eq3316}). Similarly to the case $b\geq7$ and $a+1>b-1$, we first compute the invariant factors of $M$ by applying row and column operations and then determine $\mathrm{Ext}_A^2(K_{\lambda}F,K_{\mu}F)$  using (\ref{eq333}) and (\ref{eq334}). 

In the following, we denote by $F_n=F_n(a,b)$ the $n-th$ operation applied to matrix $M$ and $M_n=M_n(a,b):=F_nM_{n-1}$, where $M_0:=M$. Hence, from (\ref{eq335}), it follows that $M_n$ is given by Proposition \ref{pro41} by setting $F_n$ in place of $D_n$.

\subsection{\textbf{Let} $\boldsymbol{b=3}$}
From (\ref{eq335}) - (\ref{eq3316}), it follows that
\begin{equation*}
M=M(a,b=3)=
\begin{bmatrix}
-2&0&0&a+1&-2\\
\;\;\;0&-2&-1&-1&-2
\end{bmatrix}\in\mathbb{Z}^{2\times5}.
\end{equation*}
It is easy to compute that the invariant factors of $M$ are $d_1=\pm1$ and  $d_2=\pm \mathrm{g.c.d.}\{2,a+1\}$. Therefore, (\ref{eq333}) and (\ref{eq334}) yield Theorem \ref{thm33}(1) for $b=3$.

\subsection{\textbf{Let} $\boldsymbol{b=4}$}
We consider matrix $M=M(a,b=4)$, given by  (\ref{eq335}) - (\ref{eq3316}) for $b=4$.

\subsubsection*{\textbf{Operation} $\boldsymbol{F_1}$}
Definition \ref{def43} is valid for $b=4$, thus we set $F_1:=\{D_1:b=4\}$.
We note that Definitions \ref{def49}, \ref{def413} and \ref{def416} are not valid for $b=4$.

\subsubsection*{\textbf{Operation} $\boldsymbol{F_2}$}
Definitions \ref{def420} and \ref{def423} are valid for $b=4$, therefore we set
\begin{equation*}
F_2:=\{D_5,D_6:b=4\}, \text{that is } F_2=\{F_2(w_i),F_2(u_3): i=2,3\},
\end{equation*} 
where
\begin{align*}
&F_2(w_2):\Gamma_1B_1(w_2)+\Gamma_2B_1(w_3)+\Gamma_1C_1(w_{23})+\Gamma_2C_1(w_{23}),\\
&F_2(w_3):\Gamma_2B_1(w_3)+A_1(u_2)+A_1(u_3) \text{ and } F_2(u_3):A_1(u_3)-\Gamma_1C_1(w_{23}).
\end{align*}

\subsubsection*{\textbf{Operation} $\boldsymbol{F_3}$}
We can apply Remark \ref{rem428} for $\Sigma_2\left(M_2(v_j)\right)$ to matrix $M_2$, where $j=2$ or  $j=4$, thus we set
\begin{equation*}
F_3:=\{F_3(v_j,t):j=2 \text{ or } j=4\text{ and }t=1 \text{ or } t=3\},
\end{equation*}
where
\begin{align*}
&F_3(v_2,1):\Sigma_1\left(M_2(v_2)\right)+2\Sigma_2\left(M_2(v_4)\right), F_3(v_2,3):\Sigma_3\left(M_2(v_2)\right)-\Sigma_2\left(M_2(v_2)\right),\\
&F_3(v_4,1):\Sigma_1\left(M_2(v_4)\right)-\Sigma_2\left(M_2(v_4)\right) \text{ and } F_3(v_4,3):\Sigma_3\left(M_2(v_4)\right)+2\Sigma_2\left(M_2(v_2)\right).
\end{align*}
We note that $F_3$ corresponds to operations $D_7$ and $D_9$  for $b=4$. Thus, $F_3$ is also given by Definitions \ref{def429} and \ref{def438} for $b=4$, applying the necessary modifications.

\subsubsection*{\textbf{Operation} $\boldsymbol{F_4}$}
To matrix $M_3$, we apply Remark \ref{rem428} for $\Sigma_3\left(M_3(v_2)\right)$ and a remark analogous to \ref{rem428} for $\Sigma_1\left(M_3(v_2)\right)$, with $-2$ in place of $(-1)^i$, concerning the even non-zero entries of  $\Gamma_*$. Thus, we set
\begin{equation*}
F_4:=\{F_4(v_1), F_4(v_3,t), F_4(v_4,1), F_4(v_4,3):t=1,2,3\},
\end{equation*}
where
\begin{align*}
&F_4(v_1):M_3(v_1)+2\Sigma_3\left(M_3(v_2)\right)+\Sigma_1\left(M_3(v_2)\right),\\ &F_4(v_3,t):\Sigma_t\left(M_3(v_3)\right)+\theta_t\Sigma_3\left(M_3(v_2)\right), 1\leq t\leq 3, \;\theta_1=-(a+1), \;\theta_2=2,\;\theta_3=1,\\
&F_4(v_4,1):\Sigma_1\left(M_3(v_4)\right)+(a+1)\Sigma_3\left(M_3(v_2)\right) \text{ and }\\
&F_4(v_4,3):\Sigma_3\left(M_3(v_4)\right)-\Sigma_1\left(M_3(v_2)\right).
\end{align*}

\subsubsection*{\textbf{Operation} $\boldsymbol{F_5}$}
We set $F_5:=\{F_5(w_2):\Gamma_2B_4(w_2)+3\Gamma_1B_4(w_2)+\Gamma_1B_4(w_3)\}$.

\subsubsection*{\textbf{Operation} $\boldsymbol{F_6}$}
We set $F_6:=\{F_6(v_3):\Sigma_1\left(M_5(v_3)\right)+\Sigma_3\left(M_5(v_3)\right)\}$.

\subsubsection*{\textbf{Operation} $\boldsymbol{F_7}$}
We consider matrix $M_6$ and apply the appropriate interchanges of rows and columns, and the necessary  changes of sign to rows and columns so that $M_6$  is transformed into the equivalent matrix
\begin{equation*}
M_7=
\begin{bmatrix}
I_4 & & & &\\
& 2& a+2& 0&\\
& 0& 0& 3&\\
& 0& 0& a+1&\\
&  &  &  &  O_{1\times3}
\end{bmatrix}.
\end{equation*}

According to Proposition \ref{pro446}, the invariant factors of matrix $M$ are $d_i=\pm1$, for $1\leq i\leq 5$,
$d_6=\pm\mathrm{g.c.d.}\{6, 2(a+1), 3(a+2), (a+1)(a+2)\}$ and $d_7=d_8=0$. Therefore, (\ref{eq333}) and (\ref{eq334}) yield Theorem \ref{thm33}(2).

\subsection{\textbf{Let} $\boldsymbol{b=5}$}
We consider matrix $M=M(a,b=5)$, given by (\ref{eq335}) - (\ref{eq3316}) for $b=5$.

\subsubsection*{\textbf{Operation} $\boldsymbol{F_n, n=1,2,3}$}
Definitions \ref{def43}, \ref{def49} and \ref{def413} are valid for $b=5$, thus for $n=1,2,3$, we set $F_n:=\{D_n:b=5\}$.

\subsubsection*{\textbf{Operation} $\boldsymbol{F_4}$}
For $b=5$, Definition \ref{def416} is valid only for $i=3$. 
Thus, we set $F_4:=\{D_4:b=5,\; i=3\}$. 
We notice that matrix $M_4$ differs from the corresponding matrix of the case $b\geq7$ only at $\Sigma_r\left(M_4(v_3)\right)$ and $\Sigma_t\left(M_4(v_4)\right)$, where $r=3,4$ and $t=1,2$, because $F_4$, unlike $D_4$, does not affect the corresponding columns of $M_3$.

\subsubsection*{\textbf{Operation} $\boldsymbol{F_n, n=5,...,10}$}
Definitions \ref{def420}, \ref{def423}, \ref{def429}, \ref{def432}, \ref{def438} and \ref{def442} are valid for $b=5$, thus for $n=5,\ldots,10$, we set $F_n:=\{D_n:b=5\}$. 
We notice that for $n=5,\ldots,10$, matrix $M_n$ differs from the corresponding matrix of the case $b\geq7$ only at $\Sigma_r\left(M_n(v_3)\right)$ and $\Sigma_t\left(M_n(v_4)\right)$, where $r=3,4$ and $t=1,2$.\\

In the following, we define and apply further operations to simplify  $\Sigma_r\left(M_{10}(v_3)\right)$ and $\Sigma_t\left(M_{10}(v_4)\right)$, where $r=3,4$ and  $t=1,2$.

\subsubsection*{\textbf{Operation} $\boldsymbol{F_{11}}$}
We consider matrix $M_{10}$ and focus on the rows containing non-zero entries at the columns in question.

\begin{definition}\label{def51}
We define operation $F_{11}$  as 
\begin{equation*}
F_{11}:=\{F_{11}(u_i),F_{11}(w_2,l),F_{11}(w_4,t):\ 2\leq i,t\leq 3, \ 1\leq l\leq 3\},
\end{equation*}
where
\begin{align*}
&F_{11}(u_2): A_{10}(u_2)-\Gamma_3C_{10}(w_{34}), \ F_{11}(u_3):A_{10}(u_3)+\Gamma_3B_{10}(w_4),\\
&F_{11}(w_2,1):\Gamma_1B_{10}(w_2)-\Gamma_2B_{10}(w_3),\\ 
&F_{11}(w_2,2):\Gamma_2B_{10}(w_2)+\Gamma_3B_{10}(w_2)-3\Gamma_3C_{10}(w_{34}),\\
&F_{11}(w_2,3):\Gamma_3B_{10}(w_2)+\Gamma_1B_{10}(w_4)-3\left(\Gamma_1C_{10}(w_{23})+\Gamma_2B_{10}(w_3)+\Gamma_3C_{10}(w_{34})\right),\\
&F_{11}(w_4,2):\Gamma_2B_{10}(w_4)+\Gamma_1B_{10}(w_4)-3\Gamma_1C_{10}(w_{23})\text{ and }\\
&F_{11}(w_4,3):\Gamma_3B_{10}(w_4)+\Gamma_1C_{10}(w_{23}). 
\end{align*}
\end{definition}

\subsubsection*{\textbf{Operation} $\boldsymbol{F_{12}}$}
We consider matrix $M_{11}$ and focus on $\Sigma_r\left(M_{11}(v_3)\right)$ and $\Sigma_t\left(M_{11}(v_4)\right)$, where $r=3,4$ and $t=1,2$.

\begin{definition}\label{def52}	
We define operation $F_{12}$ as
\begin{equation*}
F_{12}:=\{F_{12}(v_3,3),F_{12}(v_4,1),F_{12}(v_4,2)\},
\end{equation*}
where
\begin{align*}
&F_{12}(v_3,3):\Sigma_3(M_{11}(v_3))-\Sigma_4(M_{11}(v_3)),\\
&F_{12}(v_4,1):\Sigma_1(M_{11}(v_4))-\Sigma_2 (M_{11}(v_4))+\Sigma_3(M_{11}(v_3))-\Sigma_4 (M_{11}(v_3)) \text{ and }\\
&F_{12}(v_{4},2):\Sigma_2(M_{11}(v_4)) - \Sigma_3(M_{11}(v_3))+\Sigma_4(M_{11}(v_3)).
\end{align*}
\end{definition}

\subsubsection*{\textbf{Operation} $\boldsymbol{F_{13}}$}
We consider matrix $M_{12}$ and apply the appropriate interchanges of rows and columns, and the necessary changes of sign to rows and columns so that $M_{12}$ is transformed into the equivalent matrix
\begin{equation*}
M_{13}=
\begin{bmatrix}
I_{10}& & & &\\
&2&a+1&0&\\
&0&0&3&\\
&0&0&a+2&\\
& & & &O_{8\times4}
\end{bmatrix}.
\end{equation*}

According to Proposition \ref{pro446}, the invariant factors of matrix $M$ are $d_i=\pm1$, for $1\leq i\leq 11$, $d_{12}=\pm\mathrm{g.c.d.}\{6, 2(a+2), 3(a+1), (a+1)(a+2)\}$ and $d_i=0$, for $13\leq i\leq17$. The above, (\ref{eq333}) and (\ref{eq334}), yield Theorem \ref{thm33}(3).

\subsection{\textbf{Let} $\boldsymbol{b=6}$}
We proceed exactly as in case $b\geq7$, setting $b=6$ and modifying only the $4-th$ operation as follows: We consider matrix $M_3=M_3(a,b=6)$, given by Proposition \ref{pro41}, by setting $n=3$, $b=6$ and $F_3:=\{D_3:b=6\}$ in place of $D_3$. We notice that $M_3$ satisfies Corollary \ref{cor415} for $b=6$ and $(i,l)\notin\{(2,3),(5,2)\}$. Thus, we set 
\begin{equation*}
F_{4.1}:=\{D_4:b=6 \text{ and } (i,l)\notin\{(2,3),(5,2)\}\},
\end{equation*}
where  $D_4$ is given by Definition \ref{def416}. Now, we obtain matrix $M_{4.1}:=F_{4.1}M_3$.  Obviously, we have the following
\begin{corollary}\label{cor53}
The unique change caused to matrix $M_3$ by applying operation $F_{4.1}$ is the one described by Remark $\mathrm{\ref{rem417}}$ by setting $b=6$ and $(i,l)\notin\{(2,3),(5,2)\}$.
\end{corollary}

In matrix $M_{4.1}$, we notice that for $(i,l) \in\{(2,3),(5,2)\}$, $\Gamma_lB_{4.1}(w_i)= \Gamma_lB_3(w_i)$ and $\Gamma_lB_3(w_i)$ satisfies Corollary \ref{cor415} $(1)$ and $(3)$. In addition, each of $\Gamma(2,3):=\Gamma_3C_{4.1}(w_{34})$ and $\Gamma(5,2):=\Gamma_2C_{4.1}(w_{34})$ contains a unique non-zero entry equal to $\pm1$, located at $\Sigma_4 \left(M_{4.1}(v_3)\right)=\Sigma_4\left(M_3(v_3)\right)$ and $\Sigma_2 \left(M_{4.1}(v_5)\right)=\Sigma_2\left(M_3 (v_5)\right)$, respectively. Thus, we set 
\begin{equation*}
F_{4.2}:=\{D_4: b=6 \text{ and } (i,l)\in\{(2,3),(5,2)\}\}.
\end{equation*}

Now, we obtain matrix $M_{4.2}:=F_{4.2}M_{4.1}=
F_{4.2}F_{4.1}M_3$. Obviously, we have the following
\begin{corollary}\label{cor54}
The unique change caused to matrix $M_{4.1}$ by applying operation  $F_{4.2}$ is the one described by Remark $\mathrm{\ref{rem417}}$ by setting $b=6$ and $(i,l)\in\{(2,3),(5,2)\}$.
\end{corollary} 
Now, we set $F_4:=F_{4.2}F_{4.1}$ and obtain matrix $M_4:=F_4M_3$. From Corollaries \ref{cor53} and \ref{cor54}, it follows that the unique change caused to matrix $M_3$ by applying operation $F_4$ is the one described by Remark \ref{rem417} by setting $b=6$.

We note that matrices $A_n(u_i,v_j)$, $B_n (w_i,v_j)$ and $C_n(w_{ij},v_h)$, where $n=1,...,10$, and $M_{11}$ are obtained by setting $b=6$ in the relations that give the corresponding matrices for the case $b\geq7$. Finally, the result for $b=6$ is given by Theorem \ref{thm33}(1).

\end{document}